\documentclass[12pt]{amsart}

\usepackage{amssymb, amscd, txfonts}
\usepackage{graphicx}

\numberwithin{equation}{section}

\newtheorem{theorem}{Theorem}[section]
\newtheorem{proposition}[theorem]{Proposition}
\newtheorem{lemma}[theorem]{Lemma}
\newtheorem{corollary}[theorem]{Corollary}

\theoremstyle{definition}
\newtheorem{definition}[theorem]{Definition}
\newtheorem{example}[theorem]{Example}

\theoremstyle{remark}
\newtheorem{remark}[theorem]{Remark}

\newcommand{\N}{\mathbb{N}}
\newcommand{\Z}{\mathbb{Z}}
\newcommand{\Zp}{\mathbb{Z}_{p}}

\newcommand{\Q}{\mathbb{Q}}
\newcommand{\R}{\mathbb{R}}
\newcommand{\C}{\mathbb{C}}

\newcommand{\proj}{{\mathbb P}}

\newcommand{\SL}{{\rm SL}_2(\mathbb{Z})}
\newcommand{\Mp}{{\rm Mp}_2(\mathbb{Z})}
\newcommand{\Or}{{\rm O}^+}
\newcommand{\Ost}{\widetilde{{\rm O}}^+}
\newcommand{\SOL}{{\rm SO}^+(L)}

\newcommand{\HM}{{\rm vol}_{HM}}
\newcommand{\HMOL}{{\rm vol}_{HM}({\rm O}(L))}
\newcommand{\HMOK}{{\rm vol}_{HM}({\rm O}(K))}
\newcommand{\HMp}{{\rm vol}_{HM}^{+}}
\newcommand{\HMOLp}{{\rm vol}_{HM}({\rm O}^{+}(L))}
\newcommand{\HMOKp}{{\rm vol}_{HM}({\rm O}^{+}(K))}
\newcommand{\DL}{\mathcal{D}_{L}}
\newcommand{\FL}{\mathcal{F}_{L}}
\newcommand{\FLcpt}{\bar{\mathcal{F}_{L}}}
\newcommand{\OAL}{{\rm O}(A_{L})}

\newcommand{\RI}{\mathcal{R}_{\textrm{I}}}
\newcommand{\RII}{\mathcal{R}_{\textrm{II}}}
\newcommand{\eL}{\varepsilon(L)}
\newcommand{\epL}{\varepsilon_{p}(L)}
\newcommand{\epjL}{\varepsilon_{p,j}(L)}
\newcommand{\epnuL}{\varepsilon_{p,\nu(p)}(L)}

\begin{document}

\title[]{On the Kodaira dimension of orthogonal modular varieties}
\author[]{Shouhei Ma}
\thanks{Supported by Grant-in-Aid for Scientific Research (S) 15H05738.} 
\address{Department~of~Mathematics, Tokyo~Institute~of~Technology, Tokyo 152-8551, Japan}
\email{ma@math.titech.ac.jp}
\maketitle 

\begin{abstract}
We prove that up to scaling there are only finitely many integral lattices $L$ of signature $(2, n)$ with $n\geq21$ or $n=17$ 
such that the modular variety defined by the orthogonal group of $L$ is not of general type. 
In particular, when $n\geq108$, 
every modular variety defined by an arithmetic group for a rational quadratic form of signature $(2, n)$ is of general type. 
We also obtain similar finiteness in $n\geq9$ for the stable orthogonal groups.  
As a byproduct we derive finiteness of lattices admitting reflective modular form of bounded vanishing order, 
which proves a conjecture of Gritsenko and Nikulin.  
\end{abstract}

\setcounter{tocdepth}{1}
\tableofcontents


\section{Main results}\label{sec:intro}

It is one of classical problems in the theory of modular forms of several variables 
to determine the birational type of arithmetic quotients of Hermitian symmetric domains. 
Tai \cite{Ta}, Freitag \cite{Fr} and Mumford \cite{Mu} proved that the Siegel modular variety $\mathcal{A}_g$ is of general type in $g\geq7$, 
which first revealed the phenomenon that in higher dimension, 
modular varieties would be often of general type even for basic class of arithmetic groups, 
hence unirational case should be rare. 
Our purpose is to address this problem for modular varieties of orthogonal type. 

Let $L$ be an integral lattice of signature $(2, n)$ and ${\rm O}(L)$ be its orthogonal group. 
The Hermitian symmetric domain ${\DL}$ of type IV attached to $L$ is defined as one of the two components of the space 
\begin{equation*}\label{eqn:def type IV domain}
\{  {\C}\omega \in {\proj}(L\otimes{\C}) \; | \; (\omega, \omega)=0, (\omega, \bar{\omega})>0  \} . 
\end{equation*}
Let ${\Or}(L)$ be the subgroup of ${\rm O}(L)$ preserving ${\DL}$. 
The quotient space 
\begin{equation*}\label{eqn:def FL}
{\FL} = {\Or}(L)\backslash{\DL} 
\end{equation*}
has the structure of a quasi-projective variety of dimension $n$. 
It is invariant under scaling of $L$. 

\begin{theorem}\label{main}
Up to scaling there are only finitely many integral lattices $L$ of signature $(2, n)$ with $n\geq21$ or $n=17$ 
such that ${\FL}$ is not of general type. 
In particular, when $n\geq108$, ${\FL}$ is always of general type. 
\end{theorem}

The proof is effective: 
we will derive an explicit bound $D(n)$ determined by $n$ such that 
for primitive lattices $L$ of signature $(2, n)$, 
${\FL}$ is of general type whenever the exponent $D(L)$ of its discriminant group $A_L$ satisfies $\sqrt{D(L)} \geq D(n)$. 
(Recall that the exponent of a finite abelian group is the maximal order of its elements.) 
Asymptotically, 
\begin{equation}\label{eqn:estimate intro}
D(n) \sim \frac{3^2 \cdot 2^{2n+11} \cdot \pi^{n/2+1} \cdot e^2 }{ \Gamma(n/2+1)}. 
\end{equation}
The absence of non-general type case in large $n$ is a consequence of the convergence $D(n)\to 0$. 
The bound $n\geq108$ is obtained by computing a variant of this estimate, rather than itself (\S \ref{ssec:bound of n}). 
In this way, the logic to deduce finiteness is to show, in a quantitative manner, 
that ${\FL}$ must be of general type if the primitive lattice $L$ is ``large'',  
measuring the size of $L$ by $n$ and $D(L)$. 

As for the non-existence in higher dimension, 
the case of full orthogonal group covers that of general arithmetic group. 

\begin{corollary}
Let $V$ be a rational quadratic space of signature $(2, n)$ with $n\geq108$ and 
$\Gamma$ be an arithmetic subgroup of ${\Or}(V)$. 
The quotient space $\Gamma\backslash\mathcal{D}_{V}$ is always of general type. 
\end{corollary}

This holds because we can find a lattice $L\subset V$ that is stable under the action of $\Gamma$ 
and hence $\Gamma\backslash\mathcal{D}_{V}$ dominates $\mathcal{F}_{L}$, the latter being of general type. 

Another class of arithmetic groups that are often studied is the stable orthogonal groups ${\Ost}(L)$ for $L$ even, 
which is the kernel of ${\Or}(L)\to{\rm O}(A_L)$. 
The quotient ${\Ost}(L)\backslash{\DL}$ is a covering of ${\FL}$ (and changes under scaling). 
For them we obtain finiteness result in $n\geq9$. 

\begin{theorem}\label{stable}
There are only finitely many even lattices $L$ of signature $(2, n)$ with $n\geq9$ such that 
${\Ost}(L)\backslash{\DL}$ is not of general type. 
\end{theorem}

The study of Kodaira dimension of orthogonal modular varieties has been pioneered in the nineties by 
Kond\=o \cite{Ko1}, \cite{Ko2} and Gritsenko \cite{Gr}, whose main object was the moduli spaces of polarized $K3$ surfaces. 
They created several techniques for constructing pluricanonical forms, which were subsequently developed 
by Gritsenko-Hulek-Sankaran in the series of fundamental work \cite{G-H-S1}, \cite{G-H-S2}, \cite{G-H-S3}. 
In particular, in \cite{G-H-S1} they almost completed the $K3$ case by using quasi-pullback of the Borcherds $\Phi_{12}$ function \cite{B-K-P-SB}. 
This method gives a fairly nice bound (see also \cite{G-H-S4}, \cite{G-H-S5}, \cite{T-VA}), but can be applied only in dimension $n<26$. 
On the other hand, their second paper \cite{G-H-S3}  (originally designed for the $K3$ case before \cite{G-H-S1})
used the Gritsenko lifting \cite{Gr} and estimate of Hirzebruch-Mumford volume \cite{G-H-S2}, 
and studied for the first time a series of higher dimensional orthogonal modular varieties. 
In contrast to the quasi-pullback of $\Phi_{12}$,  
the method of \cite{G-H-S3} gives coarser bound in lower dimension but instead can be applied in any dimension. 
The proof of Theorem \ref{main} is based on a generalization of the method of \cite{G-H-S3}.

In algebraic geometry, orthogonal modular varieties also appear as the period spaces of 
(lattice-)polarized 
holomorphic symplectic manifolds. 
Theorem \ref{main} says that the moduli spaces of polarized symplectic manifolds must be of general type 
when the second Betti number is sufficiently large.  
Informally, one cannot have explicit parametrization of \textit{generic} such varieties. 
For known examples, Theorems \ref{main} and \ref{stable} cover the O'Grady's $10$-dimensional case and the $K3^{[N]}$-type case, 
proving finiteness of polarization types with non-general type moduli space. 
In particular, when $N>>0$, moduli space for $K3^{[N]}$-type is of general type for any polarization type. 
This extends the results of \cite{G-H-S4}, \cite{G-H-S5}. 
A natural question is whether there are only finitely many deformation types of polarized symplectic manifolds 
with non-general type moduli space. 
In view of Huybrechts' theorem \cite{Huy}, the gap between this problem and results as above rests on the possibility of Fujiki constant.

It is my pleasure to thank 
Valery Gritsenko, Klaus Hulek, Shigeyuki Kond\=o and Gregory Sankaran
for their valuable comments at various stages of this project.


\subsection{Structure of the proof}\label{ssec:proof outline}

We now give a coherent account of the proof. 
Let $L$ be an integral lattice of signature $(2, n)$. 
A standard approach for proving that ${\FL}$ is of general type is 
to produce pluricanonical forms on a toroidal compactification of ${\FL}$ via modular forms. 
When $n\geq9$, Gritsenko-Hulek-Sankaran \cite{G-H-S1} showed that 
there exists a projective toroidal compactification ${\FLcpt}$ of ${\FL}$ 
that has only canonical quotient singularity and has no brach divisor in the boundary. 
(In the Appendix we supplement their proof for the $0$-dimensional cusp case.) 
Furthermore, they showed that when $n\geq3$, 
every component of the ramification divisor of the projection ${\DL}\to{\FL}$ is defined by a reflection of $L$, 
in particular has ramification index $2$. 
The canonical divisor of ${\FLcpt}$ is then ${\Q}$-linearly equivalent to 
\begin{equation*}
K_{{\FLcpt}} \sim_{{\Q}} n\mathcal{L} - \Delta - B/2, 
\end{equation*}
where $\mathcal{L}$ is the ${\Q}$-line bundle of modular forms of weight $1$ (the Hodge bundle), 
$\Delta\subset{\FLcpt}$ the boundary divisor, 
and $B\subset{\FLcpt}$ the branch divisor of ${\DL}\to{\FL}$. 
The bundle $\mathcal{L}$ is big, and this is the source for proving that $K_{{\FLcpt}}$ is big. 
We view $\Delta$ and $B/2$ as obstruction for $K_{{\FLcpt}}$ to be big, 
and deal with them separately by dividing the canonical weight $n$.  

\begin{theorem}\label{cusp obstruction}
$(1)$ Let $n\geq21$ or $n=17$. 
For every lattice $L$ of signature $(2, n)$ there exists a nonzero cusp form of weight $<n$ with respect to ${\Or}(L)$. 

$(2)$ Let $4|n$ with $n\geq16$. 
For every lattice $L$ of signature $(2, n)$ there exists a nonzero cusp form of weight $n$ with respect to ${\Or}(L)$. 
\end{theorem}

\begin{theorem}\label{branch obstruction}
Fix a rational number $a>0$.  
Up to scaling there are only finitely many lattices $L$ 
of signature $(2, n)$ with $n\geq4$ such that the ${\Q}$-divisor $a\mathcal{L}-B/2$ of ${\FL}$ is not big.  
\end{theorem}

Theorem \ref{cusp obstruction} (2) is not used here. 
In Theorem \ref{branch obstruction}, sections of $m\mathcal{L}$ over ${\FL}$ always extend over ${\FLcpt}$ by the Koecher principle, 
so we may replace ${\FL}$ by ${\FLcpt}$. 

It is straightforward to derive Theorem \ref{main} from these two sub-theorems. 
Let $n'<n$ be the weight of cusp form in Theorem \ref{cusp obstruction} (1), and we apply Theorem \ref{branch obstruction} with $a=1$. 
This tells that in the range $n\geq21$ or $n=17$, for all but finitely many lattices (up to scaling),  
we can find a division 
\begin{equation*} 
K_{{\FLcpt}} \sim_{{\Q}} (n'\mathcal{L} - \Delta) +  (n''\mathcal{L}- B/2)  
\end{equation*}
such that $n'\mathcal{L} - \Delta$ is effective and $n''\mathcal{L}- B/2$ is big. 
Therefore $K_{{\FLcpt}}$ is big for those lattices $L$. 
Since ${\FLcpt}$ has canonical singularity, its desingularization is of general type. 
This proves Theorem \ref{main}. 

Theorems \ref{cusp obstruction} and \ref{branch obstruction} are independent, and both effective. 
In Theorem \ref{cusp obstruction} (1), the weight of cusp form can be taken to be $n/2+l+5$ 
where $l\leq6$ is as defined in Table \ref{table l_0}. 
In particular, it does not exceed $n/2+11$. 
In Theorem \ref{branch obstruction}, finiteness up to scaling for integral lattices is equivalent to finiteness for primitive lattices.  
Then, for primitive $L$, we show that 
$a\mathcal{L}-B/2$ is big if the exponent $D(L)$ of $A_L$ exceeds the explicit bound \eqref{eqn:bigness in terms of D(L)}: 
\begin{equation*}\label{eqn:estimate bigness intro}
\begin{split}
\sqrt{D(L)} \; & \; \geq \; g(n) \cdot (1+a^{-1})^{n-1} \cdot (n/2a) \\
& \; \sim \; \frac{3^2 \cdot 2^{2n+11} \cdot \pi^{n/2+1}}{\Gamma(n/2+1)} \cdot (1+a^{-1})^{n-1} \cdot (n/2a).  
\end{split}
\end{equation*}
The asymptotic \eqref{eqn:estimate intro} is obtained by putting $a=n/2-11$ in this bound.

For Theorem \ref{stable}, it suffices to prove finiteness for fixed $n$, in view of Theorem \ref{main}.  
We use in place of Theorem \ref{cusp obstruction} (1) the following. 

\begin{theorem}\label{cusp obstruction II}
For all but finitely many even lattices $L$ of signature $(2, n)$ with $n\geq5$ and containing $2U$,
we can find a nonzero cusp form of weight $<n$ with respect to ${\Ost}(L)$.  
\end{theorem}

Combined with Theorem \ref{branch obstruction} 
(note that $U$ is primitive and that the ramification divisor of ${\Ost}(L)$ is contained in that of ${\Or}(L)$),  
this proves finiteness of even lattices $L$ with $n\geq 9$ and containing $2U$ such that ${\Ost}(L)\backslash{\DL}$ is not of general type. 
In order to extend this to general even lattices, we use overlattice construction. 
If $L'$ is a (finite-index) overlattice of a lattice $L$, we have ${\Ost}(L)\subset {\Ost}(L')$ inside ${\Or}(L_{{\Q}})={\Or}(L'_{{\Q}})$, 
hence ${\Ost}(L)\backslash{\DL}$ dominates ${\Ost}(L')\backslash\mathcal{D}_{L'}$. 

\begin{lemma}
Let $L$ be an even lattice of signature $(2, n)$ with $n\geq8$. 
There exists an even overlattice $L'$ of $L$ containing $2U$ such that $D(L')=D(L)$. 
\end{lemma}

\begin{proof}
Recall that even overlattice $L'$ of $L$ corresponds to isotropic subgroup $G=L'/L$ of $A_L$ and $A_{L'}\simeq G^{\perp}/G$. 
By Nikulin \cite{Ni}, $L'$ contains $2U$ if $G^{\perp}/G$ has length $\leq n-3$. 
Let $A_L=\oplus_{p}A_p$ be the decomposition into $p$-parts. 
By Wall's classification \cite{Wa}, there exists a nondegenerate subgroup $A_p'$ of $A_p$ of the same exponent as $A_p$ and length $\leq 2$. 
We have $A_p=A_p'\oplus(A_p')^{\perp}$. 
If $G_p$ is a maximal isotropic subgroup of $(A_p')^{\perp}$, 
$G_p^{\perp}\cap (A_p')^{\perp}/G_p$ is anisotropic and so has length $\leq3$. 
We then put $G=\oplus_{p}G_p$. 
\end{proof}

By this lemma, we see that for even lattices $L$ at each $n\geq9$, 
${\Ost}(L)\backslash{\DL}$ must be of general type if $D(L)$ exceeds some bound. 
Since $|A_L|\leq D(L)^{n+2}$, Theorem \ref{stable} follows from finiteness of class number.  
(For ${\Ost}(L)$ the bound of $|A_L|$ and $n$ can be improved: see \cite{Ma} for detail.) 

Theorems \ref{main} and \ref{stable} are thus reduced to Theorems \ref{cusp obstruction}, \ref{branch obstruction} and \ref{cusp obstruction II}. 
Theorems \ref{cusp obstruction} and \ref{cusp obstruction II} are proven in \S \ref{sec:cusp obstruction} via 
the Gritsenko-Borcherds additive lifting \cite{Gr}, \cite{Bo}. 
For Theorem \ref{cusp obstruction} we use an explicit combination of Eisenstein series, 
and for Theorem \ref{cusp obstruction II} we apply a recent result of Bruinier-Ehlen-Freitag \cite{B-E-F}. 
The proof of Theorem \ref{branch obstruction} occupies \S \ref{sec:bigness setup} -- \S \ref{sec:vol estimate II}. 
In \S \ref{sec:bigness setup}  
we relate the problem to the comparison of Hirzebruch-Mumford volume between ${\FL}$ and its branch divisors, 
generalizing an argument of \cite{G-H-S3}. 
This volume ratio will be estimated in \S \ref{sec:vol estimate I} and \S \ref{sec:vol estimate II} for primitive $L$. 
In \S \ref{sec:vol estimate I} we give an estimate for each component of the branch divisor, 
and in \S \ref{sec:vol estimate II} we take their sum over all components. 
The proof of Theorems \ref{main} and \ref{stable} will be thus completed at the end of \S \ref{sec:vol estimate II} 
except the bound $n\geq108$. 

\S \ref{sec:explicit computation} is devoted to some explicit calculation.   
In \S \ref{ssec:bound of n} we derive the bound $n\geq108$ by refining the bound \eqref{eqn:estimate intro} for a particular class of lattices. 
In \S \ref{ssec:odd unimodular} we work out the odd unimodular lattices as a typical example of transition of Kodaira dimension. 
In the Appendix we prove that toroidal compactification has canonical singularity over the $0$-dimensional cusps when the fans are chosen regular. 
This result was first found by Gritsenko-Hulek-Sankaran \cite{G-H-S1} and is one of the basis of the present article, 
but their proof needs to be modified. 

In the rest of the introduction, we explain another direct consequences of 
Theorems \ref{cusp obstruction} and \ref{branch obstruction}.

\subsection{Special orthogonal group}\label{ssec:SO(L)}

Let ${\SOL}$ be the subgroup of ${\Or}(L)$ consisting of isometries of determinant $1$. 
When $n$ is odd, ${\Or}(L)$ is generated by ${\SOL}$ and $-1$, so the quotient ${\SOL}\backslash{\DL}$ is the same as ${\FL}$. 
On the other hand, when $n$ is even, ${\SOL}$ contains no reflection nor its composition with $-1$, 
so the projection ${\DL}\to {\SOL}\backslash{\DL}$ is unramified in codimension $1$. 
Furthermore, canonical forms on smooth projective models of ${\SOL}\backslash{\DL}$ correspond to cusp forms of weight $n$ 
with respect to ${\SOL}$ (cf.~\cite{G-H-S1}, \cite{Fr}). 
Theorem \ref{cusp obstruction} implies the following. 

\begin{corollary}
(1) Let $n\geq22$ be even. 
Then ${\SOL}\backslash{\DL}$ is of general type for every lattice $L$ of signature $(2, n)$. 

(2) Let $4|n$ with $n\geq16$. 
For every lattice $L$ of signature $(2, n)$, smooth projective models of ${\SOL}\backslash{\DL}$ have positive geometric genus. 
In particular, ${\SOL}\backslash{\DL}$ has nonnegative Kodaira dimension for $n=16, 20$. 
\end{corollary}

\subsection{Reflective modular forms}\label{ssec:GN conj}

Let $n\geq3$. 
A modular form $F$ on ${\DL}$ with respect to some $\Gamma < {\Or}(L)$ and a character 
is said to be \textit{reflective}  if ${\rm div}(F)$ is set-theoretically contained in the ramification divisor of ${\DL}\to{\FL}$. 
If $F$ has weight $\alpha$ and every component of ${\rm div}(F)$ has multiplicity $\leq \beta$, 
we say (temporarily) that $F$ has slope $\leq\beta/\alpha$. 
In that case, taking the average product of $F$ over $\Gamma\backslash{\Or}(L)$, 
we see that the ${\Q}$-divisor $\beta(B/2)-\alpha\mathcal{L}$ of ${\FL}$ is ${\Q}$-effective. 
Hence $(\alpha/\beta)\mathcal{L}-B/2$ cannot be big by the Koecher principle. 
For every $r\geq \beta/\alpha$, $r^{-1}\mathcal{L}-B/2$ is not big too. 
Theorem \ref{branch obstruction} implies the following. 

\begin{corollary}\label{bounded slope}
Let $r>0$ be a fixed rational number.  
Then up to scaling there are only finitely many lattices $L$ of signature $(2, n)$ with $n\geq4$ 
which carries a reflective modular form of slope $\leq r$.  
In particular, for a fixed natural number $\beta$, there are up to scaling only finitely many lattices $L$ with $n\geq4$ 
which carries a reflective modular form of vanishing order $\leq \beta$.  
\end{corollary}


Gritsenko and Nikulin \cite{G-N} defined \textit{Lie reflective modular forms} 
as reflective modular forms of vanishing order $\leq1$ with some conditions on the Fourier coefficients. 
Their motivation comes from the theory of generalized Kac-Moody algebras. 
They conjectured that the set of lattices possessing such a modular form is finite up to scaling (\cite{G-N} Conjecture 2.5.5). 
Corollary \ref{bounded slope} gives a positive answer in $n\geq4$: 

\begin{corollary}\label{Lie reflective}
Up to scaling there are only finitely many lattices $L$ of signature $(2, n)$ with $n\geq4$ which carries a Lie reflective modular form.  
\end{corollary}

In the singular weight case, reflective modular forms are classified in \cite{Sch1}, \cite{D-H-S}, \cite{Sch2} 
for a certain class of simple lattices.  


\section{Convention}\label{sec:convention}

We summarize basic definitions. 
By an (integral) \textit{lattice} $L$ we mean a free ${\Z}$-module of finite rank equipped with a nondegenerate symmetric bilinear form 
$(\: , ): L\times L\to {\Z}$. 
The lattice $L$ is said to be \textit{even} if $(l ,l)\in2{\Z}$ for every $l\in L$. 
The scaling $L(a)$ of a lattice $L$ by a natural number $a\geq1$ has the same underlying ${\Z}$-module as $L$, with the pairing multiplied by $a$. 
A lattice $L$ is said to be \textit{primitive} if it is not isometric to a scaling of any other lattice. 
A vector $l\in L$ is said to be \textit{primitive} if $L/{\Z}l$ is free. 
For such $l$, the positive generator of the ideal $(l, L)$ of ${\Z}$ is denoted by ${\rm div}(l)$. 
When $(l, l)\ne0$, the orthogonal splitting $L={\Z}l\oplus (l^{\perp}\cap L)$ holds if and only if ${\rm div}(l)=|(l, l)|$. 
The rank $2$ hyperbolic even unimodular lattice is called the \textit{hyperbolic plane} and will be denoted by $U$. 

The dual lattice of a lattice $L$ is written as $L^{\vee}$. 
The quotient group $A_L=L^{\vee}/L$ is called the \textit{discriminant group}. 
Its length is denoted by $l(A_L)$.  
$A_L$ is equipped with a natural ${\Q}/{\Z}$-valued symmetric bilinear form. 
When $L$ is even, this symmetric form comes from the ${\Q}/2{\Z}$-valued quadratic form $A_L\to{\Q}/2{\Z}$, 
$l+L \mapsto (l, l)+2{\Z}$,  
which we call the \textit{discriminant form} of $L$. 
In some literatures, scaling of this form by $1/2$ is called the discriminant form. 
The kernel of the natural map ${\Or}(L)\to{\OAL}$ is denoted by ${\Ost}(L)$ and called the \textit{stable orthogonal group}. 

The \textit{genus} of a lattice $L$ is the set of lattices $L'$ of the same signature as $L$ 
such that $L\otimes{\Zp}\simeq L'\otimes{\Zp}$ for every $p$. 
By the Hasse-Minkowski theorem, there is no loss of generality in assuming that $L'$ is contained in $L_{{\Q}}$. 
By Nikulin \cite{Ni}, two even lattices of the same signature are in the same genus if and only if their discriminant forms are isometric. 
Two lattices $L'$, $L''$ on $L_{{\Q}}$ are said to be \textit{properly equivalent} 
if $\gamma(L')=L''$ for some $\gamma\in{\rm SO}(L_{{\Q}})$. 
If we require only $\gamma\in{\rm O}(L_{{\Q}})$, this is equivalent to $L'\simeq L''$ (abstractly isometric).

Let $L$ be a lattice of signature $(2, n)$ with $n\geq3$. 
Let $\mathcal{O}(-1)\to {\DL}$ be the restriction of the tautological bundle over ${\proj}(L_{\C})$. 
The complement of the zero section in $\mathcal{O}(-1)$ is identified with 
the affine cone $\mathcal{D}_{L}^{\bullet}$ over ${\DL}$ (the vertex removed). 
A modular form of weight $k$ with respect to a finite-index subgroup $\Gamma$ of ${\Or}(L)$ is a 
$\Gamma$-invariant holomorphic section of $\mathcal{O}(-k)$. 
It corresponds to a $\Gamma$-invariant holomorphic function on $\mathcal{D}_{L}^{\bullet}$ 
that is homogeneous of degree $-k$ on each fiber of $\mathcal{D}_{L}^{\bullet}\to{\DL}$. 
We write $M_k(\Gamma)$ for the space of modular forms of weight $k$ with respect to $\Gamma$. 
When $\Gamma$ contains $-1$, we will consider only even weight $k$ because in that case modular forms of odd weight must be identically zero. 

Let $l\in L$ be a primitive isotropic vector, which corresponds to the $0$-dimensional rational boundary component ${\C}l$ of ${\DL}$. 
Let $M=l^{\perp}\cap L/{\Z}l$. 
Choose a vector $l'\in L_{{\Q}}$ with $(l, l')=1$, and identify $M_{{\Q}}$ with $\langle l, l'\rangle^{\perp}\cap L_{{\Q}}$. 
Let $M_{{\R}}^+$ be the positive cone in $M_{{\R}}$, i.e., 
one of the two components of $\{ m\in M_{{\R}}|(m, m)>0\}$,  
and $\mathcal{D}_l=M_{{\R}}+iM_{{\R}}^{+}$ be the associated tube domain. 
We have an embedding depending on $l'$ 
\begin{equation*}\label{eqn:tube domain realization}
\mathcal{D}_l \hookrightarrow \mathcal{D}_{L}^{\bullet}, \qquad v\mapsto l'+v-\frac{1}{2}((v, v)+(l', l'))l, 
\end{equation*}
whose image is $\{ \omega\in \mathcal{D}_{L}^{\bullet} | (\omega, l)=1 \}$ which gives a nowhere vanishing section of $\mathcal{O}(-1)$. 
This also induces an isomorphism $\mathcal{D}_l\simeq \mathcal{D}_{L}$ (tube domain realization). 
In this way, depending on the choice of $l'$, 
modular forms on ${\DL}$ are translated to holomorphic functions $F(Z)$ on $\mathcal{D}_l$. 
It is invariant under translation by a lattice $U(l)_{{\Z}}$ on $M_{{\Q}}$ (see the Appendix), 
hence admits a Fourier expansion of the form 
\begin{equation*}
F(Z) = \sum_{m\in U(l)_{{\Z}}^{\vee}} c(m) \chi^m, \qquad \chi^m=e^{2\pi i(m, Z)}.  
\end{equation*}
(This is expansion by characters on the torus $M_{{\C}}/U(l)_{{\Z}}$.) 
By the Koecher principle, we have $c(m)=0$ when $m\not\in \overline{M_{{\R}}^+}$. 
If $c(m)=0$ for all $m$ with $(m, m)=0$ at all primitive isotropic $l\in L$, this modular form is called a cusp form. 
The space of cusp forms is denoted by $S_k(\Gamma)\subset M_k(\Gamma)$.


\section{Construction of cusp form}\label{sec:cusp obstruction}

In this section we prove Theorems \ref{cusp obstruction} and \ref{cusp obstruction II} . 
We construct a desired cusp form via the Gritsenko-Borcherds lifting \cite{Gr}, \cite{Bo}. 
For Theorem \ref{cusp obstruction} we first make a reduction of lattice, 
and then construct the source cusp form explicitly using Eisenstein series. 
For Theorem \ref{cusp obstruction II}  we resort to Bruinier-Ehlen-Freitag's result \cite{B-E-F}.


\subsection{Reduction of lattice}\label{ssec:reduction cusp obstruction}

For the proof of Theorem \ref{cusp obstruction} we first simplify the given lattice using a classical reduction trick (cf.~\cite{Ge}, \cite{Vi}). 

\begin{lemma}\label{Vinberg reduction}
Let $L$ be a lattice of signature $(2, n)$. 
There exists a lattice $L'$ on $L_{{\Q}}$ such that 

(1) ${\Or}(L)\subset{\Or}(L')$ inside ${\Or}(L_{\Q})$ and 

(2) $L'$ is a scaling of a lattice $L''$ for which the $p$-component of $A_{L''}$ is $p$-elementary of length $\leq n/2+1$ for every $p$. 
\end{lemma}

\begin{proof}
This is described in  \cite{Vi} \S 8.5 (see also \cite{Ge} p.198--199). 
It is useful to observe that $L'$ is obtained by inductively taking 
$L_{i+1}=L_i+p_{i}^{-1}L_i\cap p_iL_i^{\vee}$ from $L_1=L$, 
and finally taking $L'=L_N\cap a L_N^{\vee}$. 
\end{proof}

\begin{corollary}
Let $L$ be a lattice of signature $(2, n)$ with $n\geq11$. 
There exists a lattice $L_1$ on $L_{{\Q}}$ such that ${\Or}(L)\subset{\Or}(L_1)$ 
and that $L_1$ is a scaling of an even lattice $L_2$ containing $2U$. 
\end{corollary}

\begin{proof}
Let $L'$ and $L''$ be as in the lemma. 
Let $L_2\subset L''$ be the maximal even sublattice of $L''$ and 
$L_1\subset L'$ be the corresponding sublattice of $L'$. 
Since ${\Or}(L'')\subset {\Or}(L_2)$, we have ${\Or}(L')\subset {\Or}(L_1)$ and hence ${\Or}(L)\subset {\Or}(L_1)$. 
When $L''$ is even, we have $L_2=L''$; when $L''$ is odd, $A_{L''}$ is an index $2$ quotient of an index $2$ subgroup of $A_{L_2}$. 
Hence $l(A_{L_2})\leq l(A_{L''})+2\leq n/2+3$. 
Then ${\rm rk}(L_2)-l(A_{L_2})\geq5$ by our assumption $n\geq11$. 
By Nikulin's theory (\cite{Ni} Corollary 1.10.2), $L_2$ contains $2U$. 
\end{proof}

Note that we did not make full use of the property (2) in Lemma \ref{Vinberg reduction}. 
This will be used in \S \ref{ssec:bound of n}. 

We have a natural isomorphism 
\begin{equation}\label{eqn:identify domain}
\mathcal{D}_{L}^{\bullet} = \mathcal{D}_{L_1}^{\bullet} \simeq \mathcal{D}_{L_2}^{\bullet}, 
\end{equation}
where the first comes from the equality $L_{\Q}=(L_1)_{\Q}$ 
and the second from the identification $L_1=L_2$ as ${\Z}$-modules. 
The inclusion ${\Or}(L)\subset {\Or}(L_1)\simeq {\Or}(L_2)$ is compatible with this isomorphism.  
Note that the induced isomorphism $\mathcal{D}_{L}\simeq\mathcal{D}_{L_2}$ preserves the rational boundary components.

\begin{lemma}
Let $F$ be a cusp form on $\mathcal{D}_{L_2}$ with respect to ${\Or}(L_2)$. 
Via \eqref{eqn:identify domain}, $F$ gives a cusp form on ${\DL}$ of the same weight with respect to ${\Or}(L)$. 
\end{lemma}

\begin{proof}
We check that $F$ is still a cusp form for ${\Or}(L_1)$. 
Let $l, l', M$ be as in the last paragraph of \S \ref{sec:convention} for $L_2$. 
For $L_1=L_2(a)$ we use $l'/a\in (L_1)_{{\Q}}$ in place of $l'\in (L_2)_{{\Q}}$. 
Then the tube domain realization of $\mathcal{D}_{L_1}$ differs from that of $\mathcal{D}_{L_2}$ by scalar multiplication by $a$, 
both on $M_{{\C}}$ and  $\mathcal{D}_{L_2}^{\bullet}$. 
Hence if we view $U(l)_{{\Z}}^{\vee}\subset M(a)_{{\Q}}$ naturally, the Fourier expansion of $F$ for $l, l'/a, L_1$ 
is multiplication by $a^k$ of the one for $l, l', L_2$. 
\end{proof}

In this way, for the proof of Theorem \ref{cusp obstruction}, we may (and do) assume in the rest of this section that 
$L$ is even and contains $2U$.


\subsection{Lifting}\label{ssec:lift}

Gritsenko-Borcherds additive lifting \cite{Gr}, \cite{Bo}, 
essentially equivalent to that of Oda \cite{Od} and Rallis-Schiffmann \cite{R-S} in a common situation, 
is a lifting from modular forms of one variable to orthogonal modular forms. 
We assume throughout that $L$ is an even lattice of signature $(2, n)$ with $n\geq3$ and contains $2U$. 
We fix an embedding $2U\hookrightarrow L$ and write $L$ in the form $L=2U\oplus K$ with $K$ negative-definite of rank $n-2$. 
We put $M=U\oplus K$. 
As explained in \S \ref{sec:convention}, via the splitting $L=U\oplus M$ 
we can identify ${\Ost}(L)$-modular forms with holomorphic functions $F$ on the tube domain $M_{{\R}}+iM_{{\R}}^+$. 
The lattice of parallel translation coincides to $M$, so the Fourier expansion has the form 
$F(Z) = \sum_{m} c(m) \chi^m$ 
where $m\in M^{\vee}\cap \overline{M_{{\R}}^+}$ (see \cite{Gr} \S 2). 

Let ${\Mp}$ be the metaplectic double cover of ${\SL}$. 
It is well-known that ${\Mp}$ is generated by the two elements 
\begin{equation*}
S=\left( \begin{pmatrix}0&-1\\ 1&0\end{pmatrix}, \sqrt{\tau} \right), \quad 
T=\left( \begin{pmatrix}1&1\\ 0&1\end{pmatrix}, 1 \right).  
\end{equation*}
Let ${\C}[A_L]$ be the group ring over $A_L$. 
If $\lambda\in A_L$, we write $\mathbf{e}_{\lambda}\in{\C}[A_L]$ for the corresponding basis vector. 
The Weil representation is a unitary representation 
\begin{equation*}
\rho_L : {\Mp}\to{\rm GL}({\C}[A_L]) 
\end{equation*} 
defined by 
\begin{eqnarray*}
\rho_L(T)(\mathbf{e}_{\lambda}) & = & e((\lambda, \lambda)/2)\mathbf{e}_{\lambda}, \\ 
\rho_L(S)(\mathbf{e}_{\lambda}) & = & \frac{\sqrt{-1}^{n/2-1}}{\sqrt{|A_L|}}\sum_{\mu\in A_L}e(-(\lambda, \mu))\mathbf{e}_{\mu}.
\end{eqnarray*}
Here $e(x)={\exp}(2\pi ix)$ for $x\in{\Q}/{\Z}$. 
The orthogonal group ${\OAL}$ of $A_L$ acts on ${\C}[A_L]$ by permuting the standard basis vectors $\mathbf{e}_{\lambda}$. 

\begin{lemma}\label{lem:Weil rep and O(A)action}
The permutation representation of ${\OAL}$ on ${\C}[A_L]$ commutes with the Weil representation. 
\end{lemma}

\begin{proof}
It suffices to check that  
\begin{equation*}\label{eqn:Weil rep and O(A)action}
\rho_L(T)\circ\gamma = \gamma\circ\rho_L(T), \qquad \rho_L(S)\circ\gamma = \gamma\circ\rho_L(S) 
\end{equation*}
for every $\gamma\in{\OAL}$. 
The first equality follows from 
\begin{equation*}
\rho_L(T)(\mathbf{e}_{\gamma\lambda}) = e((\gamma\lambda, \gamma\lambda)/2)\mathbf{e}_{\gamma\lambda} 
 = e((\lambda, \lambda)/2)\mathbf{e}_{\gamma\lambda} = \gamma(\rho_L(T)(\mathbf{e}_{\lambda})). 
\end{equation*}
The second follows from 
\begin{equation*}
\begin{split}
\sqrt{|A_L|}\sqrt{-1}^{1-n/2} \rho_L(S)(\mathbf{e}_{\gamma\lambda}) 
& = \sum_{\mu\in A_L}e(-(\gamma\lambda, \mu))\mathbf{e}_{\mu} 
= \sum_{\mu\in A_L}e(-(\lambda, \gamma^{-1}\mu))\mathbf{e}_{\mu} \\ 
& =  \sum_{\mu'\in A_L}e(-(\lambda, \mu'))\mathbf{e}_{\gamma\mu'} 
= \sqrt{|A_L|}\sqrt{-1}^{1-n/2} \gamma ( \rho_L(S)(\mathbf{e}_{\lambda})) 
\end{split}
\end{equation*}
where we put $\mu'=\gamma^{-1}\mu$. 
\end{proof}

Modular forms of type $\rho_L$ with respect to ${\Mp}$ have Fourier expansion of the form 
\begin{equation*}
f(\tau) = 
\sum_{\lambda\in A_L} \sum_{\begin{subarray}{c} n\geq0 \\ n\in(\lambda, \lambda)/2+{\Z}  \end{subarray}} 
c_{\lambda}(n)q^n\mathbf{e}_{\lambda}, \qquad  q=e^{2\pi i\tau}. 
\end{equation*}
If $l$ is an integral or half-integral weight such that $l\equiv n/2$ mod ${\Z}$, 
we write $M_l(\rho_L)$ for the space of modular forms of weight $l$ and type $\rho_L$, 
and $S_l(\rho_L)$ the subspace of cusp forms. 
By Lemma \ref{lem:Weil rep and O(A)action}, the group ${\OAL}$ acts on $M_l(\rho_L)$. 
Explicitly, if $f$ has Fourier expansion as above, then 
\begin{equation}\label{eqn:Fourier expansion under O(A)action 1}
(\gamma \cdot f)(\tau) = 
\sum_{\lambda, n} c_{\lambda}(n)q^n\mathbf{e}_{\gamma\lambda} = 
\sum_{\lambda, n} c_{\gamma^{-1}\lambda}(n)q^n\mathbf{e}_{\lambda}. 
\end{equation}
It is clear that this action preserves $S_l(\rho_L)$. 

We have a natural isomorphism ${\OAL}\simeq{\Or}(L)/{\Ost}(L)$ by Nikulin \cite{Ni}. 
Via this ${\OAL}$ also acts on $S_k({\Ost}(L))$ by the Petersson slash operator. 
Basic properties of the Gritsenko-Borcherds lifting, in a form we need, are summarized as follows. 

\begin{theorem}[Gritsenko \cite{Gr}, Borcherds \cite{Bo}]\label{thm:lifting}
Let $L$ be an even lattice of signature $(2, n)$ with $n\geq3$ containing $2U$. 
Write $L=2U\oplus K=U\oplus M$. 
Let $l$ be an integral or half-integral weight with $l\equiv n/2$ mod ${\Z}$. 
Then there exists an injective, ${\OAL}$-equivariant linear map 
\begin{equation}\label{eqn:Gritsenko lift}
S_l(\rho_L) \to S_k({\Ost}(L)), \qquad k=l+n/2-1. 
\end{equation}
If $F=\sum c(m)\chi^{m}$ is the lifting of $f=\sum c_{\lambda}(n)q^n\mathbf{e}_{\lambda}$, 
its Fourier coefficients are given by $c(0)=0$ and for $m\ne0 \in M^{\vee}\cap \overline{M_{{\R}}^+}$   
\begin{equation}\label{eqn:lift Fourier coeff}
c(m) = \sum_{\begin{subarray}{c} a\in{\N} \\ m/a\in M^{\vee} \end{subarray}} a^{k-1} c_{[m/a]}((m/a, m/a)/2), 
\end{equation}
where $[m/a]$ denotes the class in $A_M\simeq A_L$. 
\end{theorem}

Let us add a few comments, because 
some of the properties stated above are scattered or only implicit in the literatures. 

(1) In \cite{Gr} Theorem 3.1, Gritsenko constructed the lifting in the form of Jacobi lifting, 
namely a lifting from Jacobi forms of weight $k$ and index $1$ for $K(-1)$ to ${\Ost}(L)$-modular forms of the same weight. 
Since those Jacobi forms canonically correspond to modular forms of type $\rho_L$ and weight $l=k-n/2+1$ 
(see \cite{Gr} p.1187--1188), 
his lifting can be interpreted as a lifting from modular forms of type $\rho_L$. 
Borcherds (\cite{Bo} Theorem 14.3) extended the lifting in this second form to general even lattices $L$ which does not necessarily contain $2U$. 
The formula \eqref{eqn:lift Fourier coeff} is obtained by combining explicit forms of the Jacobi lifting (\cite{Gr} p.1193) and 
that of the correspondence between Jacobi forms and modular forms of type $\rho_L$ (\cite{Gr} Lemma 2.3). 
This coincides with Borcherds' calculation of Fourier expansion of his lifting 
(loc.~cit.~item 5: his notation $M$, $K$, $n$, $\lambda$, $n\lambda$, $\delta$, $m^{+}$ is read 
$L$, $M$, $a$, $l/a$, $m$, $[m/a]$, $k$ here and $z$, $z'$ are the standard basis of $U$),  
so the two liftings indeed agree. 

(2) Injectivity: in Gritsenko's construction, the Jacobi form corresponding to a cusp form $f \in S_l(\rho_L)$ is recovered as 
the $1$st Fourier-Jacobi coefficient of the lifting of $f$ 
at the $1$-dimensional cusp associated to the chosen embedding $2U\subset L$. 
Thus the lifting map \eqref{eqn:Gritsenko lift} is injective in the present case. 
(This can also be checked directly by looking the Fourier coefficients at $(1, {\Z}, K^{\vee})$.) 
It is not known whether injectivity holds in general when $L$ does not contain $2U$. 

(3) Cusp condition: the property that the lifting of a cusp form is a cusp form is established in \cite{Gr} for maximal lattices $L$. 
Indeed, the Fourier expansion  \eqref{eqn:lift Fourier coeff} shows that 
$F$ vanishes at $1$-dimensional cusps adjacent to the standard $0$-dimensional cusp, 
and when $L$ is maximal, every $1$-dimensional cusp is ${\Ost}(L)$-equivalent to such a cusp. 
(In \cite{G-H-S1} this was extended to a wider class of lattices.)
Borcherds \cite{Bo}, in his formulation, calculated the Fourier expansion of $F$ at every $0$-dimensional cusp not necessarily coming from $U$. 
From his general formula one observes that the lifting of a cusp form is a cusp form. 
(In his notation: if $m=n\lambda\in K^{\vee}$ is isotropic, then $c_{\delta}(\lambda^2/2)=c_{\delta}(0)$ is zero for all possible $(n, \lambda, \delta)$, 
so the coefficient of $\chi^m=e((m, Z))$ is zero.) 
We note that for the Oda lifting this property was proved in \cite{Od} \S 6, Corollary 2. 

(4) ${\OAL}$-equivariance: the equivariance of the lifting with respect to ${\OAL}$ is implicit in \cite{Bo} but not stated explicitly. 
For completeness let us supplement a self-contained proof in case $L$ contains $2U$. 
Let $f=\sum c_{\lambda}(n)q^n\mathbf{e}_{\lambda}$ be a cusp form of type $\rho_L$ 
and $F=\sum c(m)\chi^{m}$ be its lifting. 
Let $\gamma\in{\OAL}$ be an isometry of $A_L$. 
By \eqref{eqn:Fourier expansion under O(A)action 1} and \eqref{eqn:lift Fourier coeff} 
the lifting of $\gamma^{-1} \cdot f$ has Fourier expansion $\sum c^{\gamma}(m)\chi^{m}$ where 
\begin{equation*}
c^{\gamma}(m) = \sum_{a|m} a^{k-1}c_{\gamma[m/a]}((m/a, m/a)/2). 
\end{equation*}
Since ${\Or}(M)\to{\rm O}(A_M)={\OAL}$ is surjective by \cite{Ni}, 
we can lift $\gamma$ to an isometry of the lattice $M$, say $\hat{\gamma}\in{\Or}(M)$. 
We have $m/a\in M^{\vee}$ if and only if $\hat{\gamma}m/a\in M^{\vee}$.
Therefore  
\begin{equation*}
c^{\gamma}(m) = \sum_{a|\hat{\gamma}m} a^{k-1} c_{[\hat{\gamma}m/a]}((\hat{\gamma}m/a, \hat{\gamma}m/a)/2) = c(\hat{\gamma}m). 
\end{equation*}
On the other hand, since the factor of automorphy on ${\Or}(M)\subset{\Or}(L)$ is constantly $1$, 
the Petersson slash operator by $\hat{\gamma}$ is just the ordinary pullback of functions on $M_{{\R}}+iM_{{\R}}^+$. 
Thus the lifting of $\gamma^{-1} \cdot f$ is equal to the Petersson slash of the lifting of $f$ by $\gamma$.

\subsection{Proof of Theorem \ref{cusp obstruction}}

Let us record a consequence of Theorem \ref{thm:lifting} in a ready-to-use form. 

\begin{corollary}\label{cor:O(AL) cusp form lift}
Let $L$ be an even lattice of signature $(2, n)$ with $n\geq3$ and containing $2U$. 
If there exists a nonzero, ${\OAL}$-invariant cusp form of type $\rho_L$ and weight $l$, 
we have a nonzero cusp form of weight $l+n/2-1$ with respect to ${\Or}(L)$. 
\end{corollary} 

We are thus reduced to constructing a cusp form of type $\rho_L$ invariant under  ${\OAL}$. 
We use Eisenstein series of Bruinier-Kuss \cite{B-K}. 

Let $l>2$ be a weight with $l+n/2-1\in2{\Z}$. 
The Eisenstein series $E_{l}^{L}(\tau)$ of weight $l$ and type $\rho_L$ is defined by (\cite{B-K} \S 4) 
\begin{equation*}
E_{l}^{L}(\tau) = \frac{1}{2}\sum_{(M, \phi)} \phi(\tau)^{-2l}\cdot \rho_L(M, \phi)^{-1}(\mathbf{e}_{0}), 
\end{equation*}
where $(M, \phi)$ runs over the coset $\langle T\rangle \backslash {\Mp}$. 
This series converges normally on $\mathbb{H}$ 
and gives a modular form of type $\rho_L$ and weight $l$ whose constant term is $2\mathbf{e}_{0}$. 
It is ${\OAL}$-invariant because $\mathbf{e}_{0}$ is fixed by ${\OAL}$ 
and the ${\OAL}$-action commutes with $\rho_L$ by Lemma \ref{lem:Weil rep and O(A)action}. 
If $E_{l}^{L}(\tau)=\sum c_{\lambda, l}(n)q^n\mathbf{e}_{\lambda}$ denotes the Fourier expansion, 
it is shown in \cite{B-K} Theorem 7 that the coefficients $c_{\lambda, l}(n)$ in $n>0$ are given by 
\begin{equation*}\label{eqn:sgn Fourier coeff Eisenstein}
(-1)^{(2l-2+n)/4} \times (\textrm{nonnegative rational number}). 
\end{equation*}
Note that the Eisenstein series in \cite{B-K} are rather for the dual representation of $\rho_L$. 
But the conversion is immediate because $\rho_L^{\vee}=\rho_{L(-1)}$ under the natural identification 
${\C}[A_L]^{\vee}={\C}[A_{L(-1)}]$ induced by the basis $\mathbf{e}_{\lambda}$. 
So our $E^{L}_{l}$ is $E_l$ for $L(-1)$ in the notation of \cite{B-K}.

Let $E_6(\tau) = 1-504q-\cdots$ be the classical scalar-valued Eisenstein series of weight $6$. 

\begin{lemma}\label{lem:construct source form} 
Choose a weight $l>2$ satisfying $l+n/2 \equiv 3$ mod $4$. 
Then 
\begin{equation}\label{eqn:construct source form} 
f = E_{l}^{L} \cdot E_6 - E_{l+6}^{L} 
\end{equation}
is a nonzero, ${\OAL}$-invariant cusp form of weight $l+6$ and type $\rho_{L}$. 
\end{lemma}

\begin{proof}
The constant term of $f$ is equal to $1\cdot 2\mathbf{e}_0 - 2\mathbf{e}_0=0$, so $f$ is a cusp form. 
Since $E_{l}^{L}$ and $E_{l+6}^{L}$ are ${\OAL}$-invariant, so is $f$. 
To see the nonvanishing of $f$, we observe that the Fourier coefficient of $f$ at $q\mathbf{e}_0$ is calculated as  
\begin{equation}\label{eqn:Fourier coeff Eisenstein combi}
1\cdot c_{0,l}(1) -504\cdot2 -c_{0,l+6}(1). 
\end{equation}
By our choice of $l$, we have $c_{0,l}(1)\leq 0$ and $c_{0,l+6}(1)\geq 0$. 
Therefore \eqref{eqn:Fourier coeff Eisenstein combi} is nonzero, whence $f$ does not vanish. 
\end{proof}

According to the congruence of $n$ modulo $8$, 
the minimal weight $l>2$ satisfying $l+n/2\equiv 3$ mod $4$ is as in Table \ref{table l_0}. 
In particular, $l\leq6$. 

\begin{table}[h]
\caption{}\label{table l_0}
\begin{center} 
\begin{tabular}{c|c|c|c|c|c|c|c|c}
$n$ mod $8$ & $0$ & $1$    & $2$ & $3$     &  $4$ & $5$    & $6$ & $7$       \\ \hline  
$l$             & $3$ & $5/2$ & $6$ & $11/2$ & $5$ & $9/2$ & $4$ & $7/2$    \\ \hline  
\end{tabular}
\end{center}
\end{table}

If $n\geq21$ or $n=17$, we have $l+6< n/2+1$ for this value of $l$. 
Thus for every even lattice $L$ in this range, 
the cusp form $f$ defined by \eqref{eqn:construct source form}  has weight $< n/2+1$. 
By Corollary \ref{cor:O(AL) cusp form lift}, when $L$ contains $2U$, the lifting of $f$ is a nonzero cusp form for ${\Or}(L)$ of weight $<n$. 
This proves Theorem \ref{cusp obstruction} (1). 

When $4|n$ with $n\geq16$, $l=n/2-5$ satisfies the congruence $l+n/2 \equiv 3$ mod $4$ and $l>2$. 
Then $f$ has weight $n/2+1$, so its lifting is a cusp form of weight $n$ for ${\Or}(L)$. 
This proves Theorem \ref{cusp obstruction} (2).

\begin{remark}
One may also try other combination such as $E_{l}^{L}E_4-E_{l+4}^{L}$, 
but their nonvanishing seems nontrivial. 
There are lattices $L$ for which $E_{l}^{L}E_4=E_{l+4}^{L}$ for the minimal weight $l$, 
e.g., $II_{2,18}$, $II_{2,18}\oplus A_1$, $II_{2,18}\oplus A_2$. 
\end{remark}

\subsection{Proof of Theorem \ref{cusp obstruction II}}

In view of Theorem \ref{cusp obstruction}, it is sufficient to see the finiteness for each $5\leq n \leq20$. 
Let $n$ be fixed. 
Bruinier-Ehlen-Freitag \cite{B-E-F} recently estimated the dimension formula for $\rho_L$-valued cusp forms in \cite{Bo2}, \cite{Sk}. 
By \cite{B-E-F} Corollary 4.7, there are only finitely many finite quadratic forms $A$ of length $\leq n-2$ such that $S_l(\rho_A)=0$ for any $l\leq3$. 
By Nikulin \cite{Ni}, even lattices $L$ of signature $(2, n)$ containing $2U$ are determined by its discriminant form $A=A_L$. 
Hence for all but finitely many such lattices $L$ we have $S_l(\rho_L)\ne 0$ for some $l\leq 3 <n/2+1$. 
By taking the lifting, this proves Theorem \ref{cusp obstruction II}. 

\begin{remark}
The dimension formula for ${\rm O}(A)$-invariant cusp forms is more complicated, 
partly involving an equivariant version of Gauss sum. 
This Gauss sum will be studied in a future paper. 
\end{remark}


\section{Reflective obstruction}\label{sec:bigness setup}

This section is the start up of the proof of Theorem \ref{branch obstruction}. 
In \S \ref{ssec:branch divisor} we classify the branch divisors of ${\FL}$. 
In \S \ref{ssec:HM vol} we show that the ${\Q}$-divisor $a\mathcal{L}-B/2$ of ${\FL}$ is big 
if a certain inequality involving Hirzebruch-Mumford volumes holds. 
These volumes (or rather their ratio) will be estimated in \S \ref{sec:vol estimate I} and \S \ref{sec:vol estimate II}. 
The proof of Theorem \ref{branch obstruction} will be completed at \S \ref{ssec:proof bigness}. 

\subsection{The branch divisor}\label{ssec:branch divisor}

Let $L$ be a lattice of signature $(2, n)$ with $n\geq3$. 
Recall that the reflection $\sigma_l$ with respect to a primitive vector $l\in L$ with $(l, l)\ne0$ is defined by 
\begin{equation*}
\sigma_l : L_{{\Q}} \to L_{{\Q}}, \quad v\mapsto v-\frac{2(v, l)}{(l, l)} l. 
\end{equation*}
When $\sigma_l \in {\Or}(L)$, namely $\sigma_l$ preserves $L$ and $(l, l)<0$, 
the vector $l$ is called a \textit{reflective vector}. 
According to \cite{G-H-S1} Corollary 2.13, 
every irreducible component of the ramification divisor of ${\DL}\to{\FL}$ is 
the fixed divisor of a reflection $\sigma_l\in{\Or}(L)$, that is, the hyperplane section   
\begin{equation*}
{\proj}(K_{{\C}})\cap{\DL} = \mathcal{D}_{K}  \qquad \textrm{where} \; \; K=l^{\perp}\cap L. 
\end{equation*}
Hence classification of the branch divisors of ${\FL}$ is equivalent to 
that of ${\Or}(L)$-equivalence classes of reflective vectors. 
The starting point is the following well-known property. 

\begin{lemma}\label{lem:split/nonsplit}
Let $l\in L$ be a primitive vector with $(l, l)<0$ and $K=l^{\perp}\cap L$ be its orthogonal complement. 
Then $l$ is reflective if and only if either we have the splitting $L={\Z}l\oplus K$ or $L$ contains ${\Z}l\oplus K$ with index $2$. 
In the first case we have $(l, l)=-{\rm div}(l)$, and in the second case $(l, l)=-2{\rm div}(l)$. 
\end{lemma}

\begin{proof}
The sublattice ${\Z}l\oplus K$ of $L$ consists of vectors $l'$ such that $(l, l)|(l, l')$. 
If we choose a vector $l_0\in L$ such that $(l, l_0)={\rm div}(l)$, 
the quotient group $L/({\Z}l\oplus K)$ is cyclic of order $-(l, l)/{\rm div}(l)$, generated by $l_0$. 
Suppose that the reflection $\sigma_l$ preserves $L$. 
Then the vector 
\begin{equation*}
l_0 -\sigma_l(l_0) = (2(l, l_0)/(l, l)) l = (2{\rm div}(l)/(l, l)) l 
\end{equation*}
is contained in $L$. 
The primitivity of $l$ implies $2{\rm div}(l)/(l, l)\in {\Z}$, 
so that $-(l, l)/{\rm div}(l)=1$ or $2$. 
Conversely, suppose that $L$ contains ${\Z}l\oplus K$ with index $\leq2$. 
By the above calculation $\sigma_l(l_0)$ is contained in $L$. 
Since ${\Z}l\oplus K$ is clearly preserved by $\sigma_l$, so is $L$. 
\end{proof}

According to this lemma, 
we shall say that a reflective vector $l$ is of \textit{split type} when $L={\Z}l\oplus K$, 
and \textit{non-split type} when ${\Z}l\oplus K$ is of index $2$ in $L$. 
We denote by ${\RI}$, ${\RII}$ 
the sets of ${\Or}(L)$-equivalence classes of reflective vectors 
of split type, non-split type respectively. 
The union ${\RI}\cup{\RII}$ corresponds to the set of irreducible components of the total branch divisor $B$ of ${\FL}$.

Each component is described as follows. 
Let $l\in L$ be a reflective vector and $B_l$ be the component of $B$ defined by $l$. 
Let $\Gamma_l < {\Or}(L)$ be the stabilizer of the vector $l$. 
We view $\Gamma_l$ as a subgroup of ${\Or}(K)$ naturally where $K=l^{\perp}\cap L$. 
Note that $\Gamma_l<{\Or}(K)$ contains $-1$ because $-\sigma_l$ fixes $l$ and restricts to $-1$ on $K$. 
The projection $\mathcal{D}_{K}\to B_l$ from the ramification divisor 
descends to a birational morphism $\Gamma_l\backslash\mathcal{D}_K \to B_l$. 
This gives the normalization of $B_l$. 

\begin{lemma}\label{stabilizer split/nonsplit}
The subgroup $\Gamma_l < {\Or}(K)$ is described as follows. 

(1) When $l$ is of split type, we have $\Gamma_l={\Or}(K)$. 

(2) When $l$ is of non-split type, $\Gamma_l$ is equal to the stabilizer of an order $2$ element of $A_K$. 
In particular, $[{\Or}(K):\Gamma_l] < 2^r$ where $r=l((A_K)_2)$. 
\end{lemma}

\begin{proof}
The split case is obvious. 
When $l$ is of non-split type, 
we choose a vector $l_0\in L$ generating $L/({\Z}l\oplus K)\simeq{\Z}/2$  
and let $k_0\in K^{\vee}$ be its orthogonal projection to $K_{{\Q}}$.   
The element $x=[k_0]\in A_K$ is of order $2$. 
For $\gamma\in{\Or}(K)$ the isometry $({\rm id}, \gamma)$ of ${\Z}l\oplus K$ preserves $L$ if and only if it fixes the element $[l_0]=([l/2], x)$ of 
$A_{{\Z}l\oplus K}$. 
Hence $\Gamma_l < {\Or}(K)$ coincides with the stabilizer of $x$, and $[{\Or}(K):\Gamma_l]=|{\Or}(K)\cdot x|$. 
The orbit ${\Or}(K)\cdot x$ is contained in the set of order $2$ elements of $A_K$. 
\end{proof}


\subsection{Hirzebruch-Mumford volume}\label{ssec:HM vol}

Let $L$ be a lattice of signature $(2, n)$ with $n>0$. 
(This will be both $L$ and $K=l^{\perp}\cap L$ in \S \ref{ssec:branch divisor}.)  
Let $\Gamma < {\Or}(L)$ be a finite-index subgroup. 
Gritsenko-Hulek-Sankaran \cite{G-H-S2} introduced 
the Hirzebruch-Mumford volume ${\HM}(\Gamma)$ of $\Gamma$ following the proportionality principle of Hirzebruch and Mumford \cite{Mu0}. 
It determines the growth of the dimension of $M_k(\Gamma)$ by (\cite{G-H-S2} Proposition 1.2) 
\begin{equation}\label{eqn:HM volume and modular forms}
{\rm dim}M_k(\Gamma) = \frac{2}{n!}{\HM}(\Gamma) k^{n} + O(k^{n-1}). 
\end{equation}
We may adopt this as an equivalent definition of ${\HM}(\Gamma)$. 
If $\Gamma' < \Gamma$ is a finite-index subgroup, we have 
\begin{equation}\label{eqn:HM cofinite subgrp}
{\HM}(\Gamma') = [\langle \Gamma, -1\rangle : \langle \Gamma', -1\rangle]\cdot {\HM}(\Gamma). 
\end{equation}

Now let $L$ be a lattice of signature $(2, n)$ with $n\geq3$ 
for which we are studying whether the ${\Q}$-divisor $a\mathcal{L}-B/2$ of ${\FL}$ is big 
where $a\in{\Q}_{>0}$. 
We relate this problem to the comparison of the Hirzebruch-Mumford volumes between ${\Or}(L)$ and the branch divisors. 
If $l\in L$ is a reflective vector with orthogonal complement $K=l^{\perp}\cap L$, 
we consider the volume ratio 
\begin{equation*}\label{eqn:def vol ratio}
{\HMp}(L, K) := \frac{{\HMOKp}}{{\HMOLp}}. 
\end{equation*}

\begin{proposition}\label{prop:big via HM volume}
Let $L$ be a lattice of signature $(2, n)$ with $n\geq3$. 
Let $a>0$ be a rational number.  
The ${\Q}$-divisor $a\mathcal{L}-B/2$ of ${\FL}$ is big if we have 
\begin{equation}\label{eqn:big via HM volume}
\sum_{[l]\in\mathcal{R}_{{\rm I}}}{\HMp}(L, K) + 2^{n+1}\cdot \sum_{[l]\in\mathcal{R}_{{\rm II}}}{\HMp}(L, K)
< \left( 1+\frac{1}{a} \right)^{1-n} \cdot \frac{2a}{n}. 
\end{equation}
\end{proposition}

\begin{proof}
By definition, $a\mathcal{L}-B/2$ is big if we could show that an estimate 
\begin{equation}\label{eqn:bigness estimate}
h^0(ka\mathcal{L}-(k/2)B) > c\cdot k^n  
\end{equation}
holds for some $c>0$ in $k>>0$, 
where $k$ runs so that both $k$ and $ka$ are even numbers. 
We shall bound the left-hand side from below. 
Choose representatives $l_1,\cdots, l_r\in L$ for ${\RI}\cup{\RII}$. 
Let $K_i=l_i^{\perp}\cap L$ and $\Gamma_i < {\Or}(K_i)$ be the stabilizer of $l_i$.   
The following is essentially proved in \cite{G-H-S3} Proposition 4.1.  

\begin{lemma}\label{quasi-pullback estimate}
When both $k$ and $ka$ are even numbers, we have 
\begin{equation}\label{eqn:quasipullback} 
h^0(ka\mathcal{L}-(k/2)B) \geq {\dim}M_{ka}({\Or}(L)) - \sum_{i=1}^{r}\sum_{j=0}^{k/2-1}{\dim}M_{ka+2j}(\Gamma_i). 
\end{equation}
\end{lemma}

\begin{proof}
For a nonnegative integer $j\geq0$, 
$H^0(ka\mathcal{L}-jB)$ is the space of ${\Or}(L)$-modular forms of weight $ka$ 
which have zero of order $\geq 2j$ along every $\mathcal{D}_{K_i}$. 
The quasi-pullback of such modular forms to $\mathcal{D}_{K_i}$ is defined by (\cite{B-K-P-SB}, \cite{G-H-S3}) 
\begin{equation}\label{eqn:quasi-pullback}
H^0(ka\mathcal{L}-jB) \to M_{ka+2j}(\Gamma_i), \qquad F\mapsto (F/(\cdot, l_i)^{2j})|_{\mathcal{D}_{K_i}}. 
\end{equation}
Note that the vanishing order of $F$ along $\mathcal{D}_{K_i}$ must be even because $\Gamma_i$ contains $-1$.  
We obtain from \eqref{eqn:quasi-pullback} the exact sequence  
\begin{equation*}
0 \to H^0(ka\mathcal{L}-(j+1)B) \to H^0(ka\mathcal{L}-jB) \to \bigoplus_{i=1}^{r}M_{ka+2j}(\Gamma_i). 
\end{equation*}
Iteration of this for $j=0,\cdots, k/2-1$ gives the desired inequality. 
\end{proof}

We study asymptotic behavior of the right-hand side of \eqref{eqn:quasipullback} 
with respect to $k$. 
For the first term, we have by \eqref{eqn:HM volume and modular forms}  
\begin{equation*}
{\dim}M_{ka}({\Or}(L)) = (2/n!) \cdot {\HMOLp} \cdot a^n \cdot k^n + O(k^{n-1}). 
\end{equation*}
The second term is estimated as 
\begin{eqnarray*}
& & \sum_{i=1}^{r}\sum_{j=0}^{k/2-1}{\dim}M_{ka+2j}(\Gamma_i)  \\ 
&=& \sum_{i=1}^{r}\sum_{j=0}^{k/2-1}\left\{ \frac{2}{(n-1)!}\cdot{\HM}(\Gamma_i)\cdot 
            (ka+2j)^{n-1}+O(k^{n-2})\right\}  \\ 
&\leq& \sum_{i=1}^{r}\frac{k}{2}\cdot\left\{ \frac{2}{(n-1)!}\cdot{\HM}(\Gamma_i)\cdot 
           (a+1)^{n-1}\cdot k^{n-1}+O(k^{n-2})\right\}.  \\ 
&=& \frac{1}{(n-1)!}\cdot \left(\sum_{i=1}^{r}{\HM}(\Gamma_i) \right) \cdot(a+1)^{n-1}\cdot k^n + O(k^{n-1}). 
\end{eqnarray*}
Comparing the coefficients of $k^n$ in these two asymptotics, 
we see that \eqref{eqn:bigness estimate} holds if 
\begin{equation*}\label{eqn:key inequality}
\sum_{i=1}^{r} \frac{{\HM}(\Gamma_i)}{{\HM}({\Or}(L))} < \left( 1+\frac{1}{a} \right)^{1-n} \cdot \frac{2a}{n}. 
\end{equation*}

It remains to classify $l_1,\cdots, l_r$ by split/non-split type. 
We have $\Gamma_i={\Or}(K_i)$ if $l_i$ is of split type. 
When $l_i$ is of non-split type, we have 
\begin{equation*}
{\HM}(\Gamma_i) = [{\Or}(K_i):\Gamma_i] \cdot {\HM}({\Or}(K_i)) < 2^{n+1}\cdot{\HM}({\Or}(K_i)) 
\end{equation*} 
by \eqref{eqn:HM cofinite subgrp} and Lemma \ref{stabilizer split/nonsplit}. 
\end{proof}

We use the relation \eqref{eqn:HM cofinite subgrp} to extend the definition formally to ${\rm O}(L)$ 
\begin{equation*}
{\HMOL} := {\HMOLp}/[{\rm O}(L):{\Or}(L)]. 
\end{equation*}
It is often convenient to consider the following variant of ${\HMp}(L, K)$ 
\begin{equation*}\label{eqn:def HM ratio nonplus}
{\HM}(L, K) := \frac{{\HMOK}}{{\HMOL}}. 
\end{equation*}
The quotient 
\begin{equation}\label{eqn:vol ratio spinor diff}
\frac{{\HM}(L, K)}{{\HMp}(L, K)} = \frac{[{\rm O}(L):{\Or}(L)]}{[{\rm O}(K):{\Or}(K)]}
\end{equation}
is equal to $1$ or $2$ or $1/2$. 


\section{Single volume estimate}\label{sec:vol estimate I}

By Proposition \ref{prop:big via HM volume}, to show that $a\mathcal{L}-B/2$ is big is reduced to 
estimating the sum of the volume ratios ${\HMp}(L, K)$. 
In order to deduce the finiteness as in Theorem \ref{branch obstruction}, 
we want to estimate it for primitive lattices $L$ in a way that reflects the ``size'' of $L$. 
This is the task of \S \ref{sec:vol estimate I} and \S \ref{sec:vol estimate II}. 
In this \S \ref{sec:vol estimate I} we estimate ${\HM}(L, K)$ for each reflective vector, 
and in the next \S \ref{sec:vol estimate II} we take their sum over all components of the branch divisor. 
The final result is Propositions \ref{prop: vol split final}, \ref{prop: vol nonsplit final} and \eqref{eqn:vol final}, 
where the dimension $n$ and the exponent $D(L)$ of $A_L$ play the role of measuring the size of $L$. 
Derivation of Theorem \ref{branch obstruction} from these estimates is done in \S \ref{ssec:proof bigness}, 
which we encourage the reader to read before going to the technical detail of the estimate. 

The central idea of \S \ref{sec:vol estimate I} and \S \ref{sec:vol estimate II} is 
to reserve the reflection of $n$ and $D(L)$ through the whole process of estimate. 
Some step in \S \ref{sec:vol estimate I} might seem indirect, 
but they are designed so that we can finally obtain a reasonable bound in \S \ref{sec:vol estimate II}. 

A word on primitivity assumption: 
in each subsection (except \S \ref{ssec:proof bigness}) we will not assume that the given lattice $L$ is primitive until the final step.  
This is not for the sake of generality, but rather is an indispensable piece in the proof for the non-split case. 

Throughout we write $D(L)$ for the exponent of the discriminant group $A_L$ of a lattice $L$. 
Clearly $D(L)$ divides $|A_L|$, and the set of prime divisors of $D(L)$ equals that of $|A_L|$.

\subsection{Volume formula}\label{ssec:vol formula}

In \cite{G-H-S2}, 
Gritsenko-Hulek-Sankaran derived an exact formula for the Hirzebruch-Mumford volume 
by carefully comparing various volume formulae related to orthogonal groups. 
Let $L$ be a lattice of signature $(2, n)$ with $n>0$. 
We write $g_{sp}^{+}(L)$ for the number of proper spinor genera in the genus of $L$. 
Since $L$ is indefinite of rank $\geq3$, 
proper spinor genus coincides with proper equivalence class (\cite{Ki} Theorem 6.3.2). 
For each prime $p$ we write $\alpha_p(L)$ for the local density of the ${\Zp}$-lattice $L\otimes{\Zp}$. 
This is also denoted as $\alpha_p(L, L)$ in literatures (cf.~\cite{Ki} p.98). 

\begin{theorem}[\cite{G-H-S2} Theorem 2.1]\label{thm:GHS formula}
Let $L$ be a lattice of signature $(2, n)$ with $n>0$. 
Then  
\begin{equation}\label{eqn:GHS formula}
{\HMOL} = 
\frac{2}{g_{sp}^{+}(L)} \cdot |A_L|^{(n+3)/2} \cdot 
\prod_{k=1}^{n+2} \pi^{-k/2}\Gamma(k/2) \prod_{p} \alpha_p(L)^{-1}, 
\end{equation}
where $\Gamma(m)$ is the Gamma function. 
\end{theorem}

Computation of the formula \eqref{eqn:GHS formula} amounts to that of 
the spinor class number $g_{sp}^{+}(L)$ and the local densities $\alpha_p(L)$. 
Below we use the notation 
\begin{equation*}
L\otimes{\Zp} = \bigoplus_{j\geq0} L_{p,j}(p^j) , \qquad {\rm rk}(L_{p,j})=n_{p,j}(L)
\end{equation*} 
for a Jordan decomposition of $L\otimes{\Zp}$. 
Each $L_{p,j}$ is a unimodular ${\Zp}$-lattice. 
When $p>2$, Jordan decomposition is unique up to isometry. 
For $p=2$, $n_{2,j}(L)$ and whether $L_{2,j}$ is even or odd are uniquely determined. 
See \cite{Ki} \S 5.3 and \cite{Ge} \S 8.3. 

Let $P$ be the set of  \label{def:P}
odd prime divisors $p$ of $D(L)$ for which $n_{p,j}(L)\leq1$ for all $j$. 
We will later use the following estimate of $g_{sp}^{+}(L)$. 

\begin{lemma}\label{lem:estimate class number}
We have 
\begin{equation*}
g_{sp}^{+}(L) \leq 4\cdot 2^{|P|}. 
\end{equation*}
\end{lemma}

\begin{proof}
This can be seen from \cite{Ca} Chapter 11.3. 
If $p\not\in P\cup\{2\}$, then $n_{p,j}(L)\geq2$ for some $j$. 
By Lemma 3.3 loc.~cit, the group $\theta({\rm SO}(L\otimes{\Zp}))$ of spinor norms of ${\rm SO}(L\otimes{\Zp})$ contains 
\begin{equation*}
\theta({\rm SO}(L_{p,j}(p^j))) = \theta({\rm SO}(L_{p,j})) = {\Z}_{p}^{\times}\cdot({\Q}_p^{\times})^2 
\end{equation*} 
for such $p$. 
By Theorem 3.1 Note 2, equality (3.35) and Lemma 3.6 (i) loc.~cit., we then have 
\begin{eqnarray*}
g_{sp}^{+}(L) 
& \leq &  \prod_{p|2D(L)} [{\Z}_{p}^{\times} : {\Z}_{p}^{\times} \cap \theta({\rm SO}(L\otimes{\Zp}))] \\ 
& \leq &  \prod_{p\in P\cup\{2\}} [{\Z}_{p}^{\times} : ({\Z}_{p}^{\times})^2 ] \\ 
& = & 4\cdot 2^{|P|}. 
\end{eqnarray*}
\end{proof}

Next we recall the formula of $\alpha_p(L)$ given in \cite{Ki} \S 5.6 (see especially p.98 and Theorem 5.6.3). 
We write $s_p(L)$ for the number of indices $j$ with $L_{p,j}\ne0$, and set  
\begin{equation*}
w_p(L) = \sum_{j} j\cdot n_{p,j}(L)\cdot \Bigl( \frac{n_{p,j}(L)+1}{2} + \sum_{k>j} n_{p,k}(L) \Bigr).  
\end{equation*}
For an even unimodular ${\Zp}$-lattice $N$ of rank $r\geq0$, 
we define $\chi(N)$ by 
$\chi(N)=0$ if $r$ is odd, 
$\chi(N)=1$ if $N\simeq (r/2)U\otimes{\Zp}$, 
and $\chi(N)=-1$ otherwise. 
For a natural number $m$ we put  
\begin{equation*}
P_p(m) = \prod_{k=1}^{m} (1-p^{-2k}) 
\end{equation*}
when $m>0$, and $P_p(0)=1$. 
Then for $p\ne2$, we have 
\begin{equation*}\label{eqn:local density odd}
\alpha_p(L) = 2^{s_p(L)-1} \cdot p^{w_p(L)} \cdot \prod_{j} P_p([n_{p,j}(L)/2]) 
\cdot \prod_{j} (1+\chi(L_{p,j})p^{-n_{p,j}(L)/2})^{-1}, 
\end{equation*}
where $j$ ranges over indices with $L_{p,j}\ne0$. 

The $2$-adic density is more complicated. 
Consider a decomposition $L_{2,j}=L_{2,j}^{+}\oplus L_{2,j}^{-}$ such that 
$L_{2,j}^{+}$ is even and $L_{2,j}^{-}$ is either $0$ or odd of rank $\leq2$. 
Put $n_{2,j}^{+}(L)={\rm rk}(L_{2,j}^{+})$. 
We also set $q(L)=\sum_{j\geq0} q_j(L)$, where 
$q_j(L)=0$ if $L_{2,j}$ is even, 
$q_j(L)=n_{2,j}(L)$ if $L_{2,j}$ is odd and $L_{2,j+1}$ is even, 
and $q_j(L)=n_{2,j}(L)+1$ if both $L_{2,j}$ and $L_{2,j+1}$ are odd. 
Here zero-lattice is counted as an even lattice. 
For an index $j$ with $L_{2,j}\ne0$,  
we define $E_{2,j}(L)$ by  
$E_{2,j}(L) = 1+\chi(L_{2,j}^+)2^{-n_{2,j}^{+}(L)/2}$ if both $L_{2,j-1}$ and $L_{2,j+1}$ are even 
and $L_{2,j}^-\nsimeq\langle\epsilon_1, \epsilon_2\rangle$ with $\epsilon_1\equiv\epsilon_2$ mod $4$, 
and $E_{2,j}(L)=1$ otherwise. 
We also let $s_2'(L)$ be the number of indices $j\geq-1$ such that 
$L_{2,j}=0$ and either $L_{2,j-1}$ or $L_{2,j+1}$ is odd. 
Then we have 
\begin{equation*}\label{eqn:local density 2}
\alpha_2(L) = 2^{n+1+w_2(L)-q(L)+s_2(L)+s_2'(L)} \cdot \prod_{j} P_2(n_{2,j}^+(L)/2) 
\cdot \prod_{j} E_{2,j}(L)^{-1}, 
\end{equation*}
where $j$ ranges over indices with $L_{2,j}\ne0$.


\subsection{Split case}\label{ssec:vol ratio I split}

We now begin the estimate of ${\HM}(L, K)$. 
We first consider the split case. 
For later purpose (\S \ref{ssec:vol ratio I nonsplit}) we will not assume until Proposition \ref{e(L) bdd split} that the lattice $L$ is primitive. 
So our initial setting is: 
$L$ is a lattice of signature $(2, n)$ with $n\geq2$, 
and $l\in L$ is a primitive vector of norm $(l, l)=-D$ such that we have the orthogonal splitting 
\begin{equation*}
L = {\Z}l \oplus K \simeq \langle -D \rangle \oplus K, \qquad K=l^{\perp}\cap L. 
\end{equation*}
We denote the prime decompositions of $D$, $D(L)$, $|A_L|$ respectively  by 
\begin{equation*}
D = \prod_{p} p^{\nu(p)}, \quad  
D(L) = \prod_{p} p^{\mu(p)}, \quad  
|A_L| = \prod_{p} |A_L|_p. 
\end{equation*}
It is clear that $\nu(p)\leq \mu(p)$. 
We use the Jordan decomposition of $L\otimes{\Zp}$ that is induced from a Jordan decomposition of $K\otimes{\Zp}$. 
Then 
\begin{equation*}
K_{p,j}\simeq L_{p,j} \qquad (j\ne \nu(p)), 
\end{equation*}
\begin{equation*}
n_{p,\nu(p)}(K)=n_{p,\nu(p)}(L)-1. 
\end{equation*}

Substituting $L$ and $K$ into the formula \eqref{eqn:GHS formula}, we obtain  
\begin{equation*}
{\HM}(L, K) \; = \; \frac{g_{sp}^{+}(L)}{g_{sp}^{+}(K)} 
\cdot \frac{\pi^{n/2+1}}{\Gamma(n/2+1)} \cdot \left( \frac{1}{D} \right)^{n/2+1} \cdot |A_L|^{-1/2} \cdot 
\prod_{p}\frac{\alpha_p(L)}{\alpha_p(K)}. 
\end{equation*}
If we put for each prime $p$ 
\begin{equation*}
a_p(L, K) := p^{-\nu(p)(n/2+1)} \cdot |A_L|_p^{-1/2} \cdot \frac{\alpha_p(L)}{\alpha_p(K)}, 
\end{equation*}
this can be rewritten as 
\begin{equation}\label{eqn:vol and ap}
{\HM}(L, K) \; = \; \frac{g_{sp}^{+}(L)}{g_{sp}^{+}(K)} 
\cdot \frac{\pi^{n/2+1}}{\Gamma(n/2+1)} \cdot \prod_{p}a_p(L, K). 
\end{equation}
Below we shall estimate $a_p(L, K)$ for each $p$. 
The case $p\nmid 2D(L)$ is easy (Lemma \ref{lem:a_p and epj split} (1)). 
When $p|D(L)$, we rearrange $a_p(L, K)$ as follows. 


\begin{lemma}\label{lem:a_p and m_{p,j}}
Let $p$ be a prime. 
For an index $j$ with $L_{p,j}\ne0$ we put 
\begin{equation*}
m_{p,j}(L) := \sum_{k\geq0} |\, k-j\, | \cdot n_{p,k}(L) - \mu(p). 
\end{equation*}
Then 
\begin{equation}\label{eqn:a_p and m_{p,j}}
a_p(L, K) = 
p^{-m_{p,\nu(p)}(L)/2} \cdot \frac{\alpha_p(L)\cdot p^{-w_p(L)}}{\alpha_p(K)\cdot p^{-w_p(K)}} \cdot p^{-\mu(p)/2}. 
\end{equation}
\end{lemma}

\begin{proof}
It suffices to check that 
\begin{equation*}
{\log}_p|A_L|_p + \nu(p)(n+2) = 2w_p(L) - 2w_p(K) + m_{p,\nu(p)}(L) + \mu(p). 
\end{equation*}
We have 
\begin{equation*}\label{eqn:a_p and m_pj 1}
{\log}_p|A_L|_p + \nu(p)(n+2) = \sum_{k\geq0} k\cdot n_{p,k}(L) + \sum_{k\geq0} \nu(p)\cdot n_{p,k}(L). 
\end{equation*}
Using the relation of $n_{p,k}(L)$ and $n_{p,k}(K)$, we can calculate 
\begin{equation*}\label{eqn:a_p and m_pj 2}
w_p(L) - w_p(K) = 
\sum_{k<\nu(p)} k\cdot n_{p,k}(L) + \nu(p) \cdot \sum_{k\geq \nu(p)}n_{p,k}(L). 
\end{equation*}
Therefore 
\begin{eqnarray*}
& & {\log}_p|A_L|_p + \nu(p)(n+2) - 2w_p(L) + 2w_p(K) \\ 
&=&  \sum_{k<\nu(p)} (\nu(p)-k)n_{p,k}(L) + \sum_{k\geq\nu(p)} (k-\nu(p))n_{p,k}(L). 
\end{eqnarray*}
\end{proof}

\begin{figure}[h]
\begin{center}
\setlength\unitlength{4.5mm}
\begin{picture}(20,8)(0,0)
\put(0,0){\line(1,0){20}}
\put(5,3){\line(1,0){3}}
\put(8,5){\line(1,0){4}}
\put(12,6){\line(1,0){3}}
\put(15,8){\line(1,0){5}}

\put(5,0){\line(0,1){3}}
\put(8,3){\line(0,1){2}}
\put(12,5){\line(0,1){1}}
\put(15,6){\line(0,1){2}}
\put(20,0){\line(0,1){8}}

\put(0,0){\dashbox{0.2}(1,5){}}
\put(19,5){\dashbox{0.2}(1,3){}}
\multiput(0,5)(0.4, 0){48}{\line(1,0){0.2}}

\put(-0.6,5){$j$}
\put(-0.6,0){$0$}
\put(20.5,8){$\mu(p)$}
\put(2,-0.9){$n_{p,0}$}
\put(17,8.7){$n_{p,\mu(p)}$}
\put(10,5.7){$n_{p,j}$}

\put(1,0){\line(1,1){5}}
\put(1.5,0){\line(1,1){5}}
\put(2,0){\line(1,1){5}}
\put(2.5,0){\line(1,1){2.5}}
\put(3,0){\line(1,1){2}}
\put(3.5,0){\line(1,1){1.5}}
\put(4,0){\line(1,1){1}}
\put(4.5,0){\line(1,1){0.5}}

\put(1,0.5){\line(1,1){4.5}}
\put(1,1){\line(1,1){4}}
\put(1,1.5){\line(1,1){3.5}}
\put(1,2){\line(1,1){3}}
\put(1,2.5){\line(1,1){2.5}}
\put(1,3){\line(1,1){2}}
\put(1,3.5){\line(1,1){1.5}}
\put(1,4){\line(1,1){1}}
\put(1,4.5){\line(1,1){0.5}}

\put(5,3){\line(1,1){2}}
\put(5.5,3){\line(1,1){2}}
\put(6,3){\line(1,1){2}}
\put(6.5,3){\line(1,1){1.5}}
\put(7,3){\line(1,1){1}}
\put(7.5,3){\line(1,1){0.5}}

\put(12,5.5){\line(1,1){0.5}}
\put(12,5){\line(1,1){1}}
\put(12.5,5){\line(1,1){1}}
\put(13,5){\line(1,1){1}}
\put(13.5,5){\line(1,1){1}}
\put(14,5){\line(1,1){3}}
\put(14.5,5){\line(1,1){3}}
\put(15,5){\line(1,1){3}}
\put(15.5,5){\line(1,1){3}}
\put(16,5){\line(1,1){3}}
\put(16.5,5){\line(1,1){2.5}}
\put(17,5){\line(1,1){2}}
\put(17.5,5){\line(1,1){1.5}}
\put(18,5){\line(1,1){1}}
\put(18.5,5){\line(1,1){0.5}}

\put(15,6.5){\line(1,1){1.5}}
\put(15,7){\line(1,1){1}}
\put(15,7.5){\line(1,1){0.5}}

\qbezier(0, 0)(2.5, -0.6)(5, 0)
\qbezier(8, 5)(10, 5.6)(12, 5)
\qbezier(15, 8)(17.5, 8.6)(20, 8)
\end{picture}
\end{center}
\caption{$m_{p,j}(L)$ (when $L\otimes{\Zp}$ is primitive)}
\label{figure:m_{p,j}}
\end{figure}
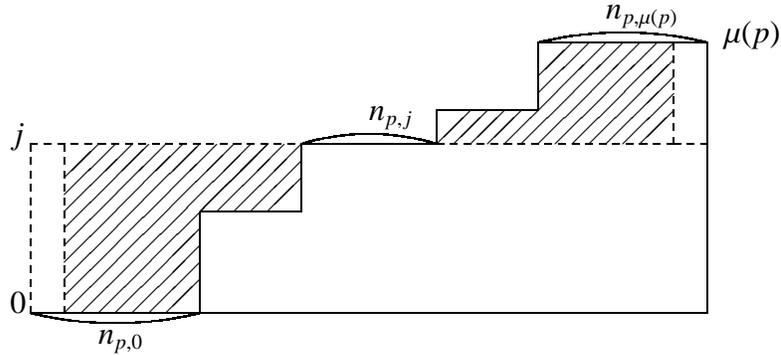 

The term $p^{-\mu(p)/2}$ that we separated in \eqref{eqn:a_p and m_{p,j}} measures the size of $L\otimes{\Zp}$. 
This will be reserved through the rest of this section. 
The number $m_{p,\nu(p)}(L)$ will be central in our estimate. 
When $L\otimes{\Zp}$ is primitive, i.e., $n_{p,0}(L)>0$, one can understand $m_{p,j}(L)$ as 
the area of the slanted region in Figure \ref{figure:m_{p,j}}.  
Let us first bound the middle term of \eqref{eqn:a_p and m_{p,j}} 
\begin{equation*}
\frac{\alpha_p(L)\cdot p^{-w_p(L)}}{\alpha_p(K)\cdot p^{-w_p(K)}} 
\end{equation*}
in the next Lemma \ref{lem:a_p and epj split}. 
The result is to be reflected in the following definition of ${\epjL}$.  

\begin{definition}\label{def: epj(L)}
Let $L$ be a lattice of signature $(2, n)$. 
Let $p$ be a prime divisor of $2D(L)$ and $j$ be an index with $L_{p,j}\ne0$. 
We set 
\begin{equation*}\label{eqn:def epj} 
{\epjL} = 
\begin{cases}
\: p^{-m_{p,j}(L)/2}(1+p^{-[n_{p,j}(L)/2]}),    & \; p\notin P\cup \{2\}, \\
\: 4 \cdot p^{-m_{p,j}(L)/2},                          & \; p\in P, \\  
\: 2^{-m_{2,j}(L)/2},                                      & \; p=2. 
\end{cases}  
\end{equation*}
\end{definition}

Note that when $2\nmid D(L)$, namely $L\otimes{\Z}_2$ is unimodular, we have $m_{2,0}(L)=0$ and hence $\varepsilon_{2,0}(L)=1$. 
Note also that $\varepsilon_{p,j}(L)$ does not depend on the choice of Jordan decomposition. 

\begin{lemma}\label{lem:a_p and epj split}
The following inequalities hold. 

(1) When $p\nmid 2D(L)$, we have 
\begin{equation*}\label{eqn:ap split 1}
a_p(L, K) \leq 1+p^{-[n/2]-1}. 
\end{equation*}

(2) When $p|D(L)$ with $p\not\in P\cup \{ 2\}$, we have  
\begin{equation*}\label{eqn:ap split 2}
a_p(L, K) \leq {\epnuL} \cdot p^{-\mu(p)/2}. 
\end{equation*}

(3) For $p\in P$ we have 
\begin{equation*}\label{eqn:ap split 4}
g_{sp}^{+}(L) \cdot \prod_{p\in P} a_p(L, K) \leq 
4 \cdot \prod_{p\in P} {\epnuL} \cdot p^{-\mu(p)/2}. 
\end{equation*}

(4) For $p=2$ we have 
\begin{equation*}\label{eqn:ap split 3}
a_2(L, K)  \leq 2^5 \cdot \varepsilon_{2,\nu(2)}(L) \cdot 2^{-\mu(2)/2}. 
\end{equation*}
\end{lemma}

\begin{proof}
(1) 
Let $p\nmid D(L)$ with $p>2$. 
In this case $a_p(L, K)$ reduces to $\alpha_p(L)/\alpha_p(K)$. 
Since both $L\otimes{\Zp}$ and $K\otimes{\Zp}$ are unimodular, 
we have $s_p(L)=s_p(K)=1$ and $w_p(L)=w_p(K)=0$. 
Then 
\begin{eqnarray*}
\frac{\alpha_p(L)}{\alpha_p(K)} &=& 
\frac{P_p([(n+2)/2])}{P_p([(n+1)/2])} \cdot \frac{1+\chi(K_{p,0})p^{-(n+1)/2}}{1+\chi(L_{p,0})p^{-(n+2)/2}} \\ 
& = & \begin{cases}
 1-\chi(L_{p,0})p^{-(n+2)/2},  & n: even, \\
 1+\chi(K_{p,0})p^{-(n+1)/2},  & n: odd,  
\end{cases} \\ 
& \leq & 1+p^{-[n/2]-1}. 
\end{eqnarray*}

(2) 
Next we consider the case $p|D(L)$ with $p>2$. 
When $n_{p,\nu(p)}(L)>1$, we have $s_p(L)=s_p(K)$.  
Then 
\begin{eqnarray*}
\frac{\alpha_p(L)\cdot p^{-w_p(L)}}{\alpha_p(K)\cdot p^{-w_p(K)}} &=& 
\frac{P_p([n_{p,\nu(p)}(L)/2])}{P_p([n_{p,\nu(p)}(K)/2])} \cdot 
\frac{1+\chi(K_{p,\nu(p)})p^{-n_{p,\nu(p)}(K)/2}}{1+\chi(L_{p,\nu(p)})p^{-n_{p,\nu(p)}(L)/2}} \\ 
& \leq & 1+p^{-[n_{p,\nu(p)}(L)/2]} 
\end{eqnarray*}
by the same calculation as in case (1). 
On the other hand, if $n_{p,\nu(p)}(L)=1$, we have $s_p(L)=s_p(K)+1$ so that   
\begin{equation}\label{eqn:alpha_p p^w_p n_pj=1}
\frac{\alpha_p(L)\cdot p^{-w_p(L)}}{\alpha_p(K)\cdot p^{-w_p(K)}} = 2. 
\end{equation}
By \eqref{eqn:a_p and m_{p,j}}, this gives the desired inequality in case $p\notin P$. 

(3) When $p\in P$, the equality \eqref{eqn:alpha_p p^w_p n_pj=1} is still valid. 
This, combined with \eqref{eqn:a_p and m_{p,j}} and Lemma \ref{lem:estimate class number}, gives the desired inequality. 
 
(4) 
Finally let $p=2$. 
Note that $L_{2,\nu(2)}$ is odd. 
It is easy to check that 
\begin{equation*}
s_2(L) - s_2(K) \leq 1,  
\end{equation*}
\begin{equation*}
s_2'(L) - s_2'(K) \leq 2,   
\end{equation*}
\begin{equation*}
q(K) - q(L) = q_{\nu(2)}(K) - q_{\nu(2)}(L) + q_{\nu(2)-1}(K) - q_{\nu(2)-1}(L) \leq -1+0 = -1,  
\end{equation*}
\begin{equation*}
\prod_{j} \frac{P_2(n_{2,j}^+(L)/2)}{P_2(n_{2,j}^+(K)/2)} 
= \frac{P_2(n_{2,\nu(2)}^+(L)/2)}{P_2(n_{2,\nu(2)}^+(K)/2)} 
\leq 1, 
\end{equation*}
\begin{equation*}
\prod_{j}\frac{E_{2,j}(K)}{E_{2,j}(L)} 
= \prod_{j=\nu(2)-1}^{\nu(2)+1}\frac{E_{2,j}(K)}{E_{2,j}(L)} 
\leq \frac{1+1}{1} \cdot \frac{1+1}{1-2^{-1}} \cdot \frac{1+1}{1} = 2^4. 
\end{equation*}
Actually, examining the cases when $s_2'(L)>s_2'(K)$ holds, we can see 
\begin{equation*}
2^{s_2'(L) - s_2'(K)}\cdot \prod_{j}E_{2,j}(K)/E_{2,j}(L) \leq 2^4. 
\end{equation*}
This gives 
\begin{equation}\label{eqn:2^7bound}
\frac{\alpha_2(L)\cdot 2^{-w_2(L)}}{\alpha_2(K)\cdot 2^{-w_2(K)}} \leq 2^5.
\end{equation}  
\end{proof}

By this lemma we obtain 
\begin{equation*}
g_{sp}^{+}(L) \cdot \prod_{p}a_p(L, K) \; < \;  
2^7 \cdot \zeta([n/2]+1) \cdot \prod_{p|D(L)} \varepsilon_{p,\nu(p)}(L) \cdot D(L)^{-1/2} 
\end{equation*}
regardless of whether $D(L)$ is even or odd. 
Substituting this into \eqref{eqn:vol and ap} gives the following intermediate estimate of ${\HM}(L, K)$.  

\begin{proposition}\label{prop:vol split single}
Let $L$ be a lattice of signature $(2, n)$ with $n\geq2$, and 
$K=l^{\perp}\cap L$ be the orthogonal complement of a reflective vector 
$l\in L$ of split type of norm $(l, l)=-D=-\prod_{p}p^{\nu(p)}$. 
Then we have 
\begin{equation*}
{\HM}(L, K) \; < \; 
\frac{1}{g_{sp}^{+}(K)} \cdot 
\frac{2^7\cdot\pi^{n/2+1}\cdot\zeta([n/2]+1)}{\Gamma(n/2+1)} \cdot D(L)^{-1/2} \cdot  
\prod_{p|D(L)} \varepsilon_{p,\nu(p)}(L). 
\end{equation*}
\end{proposition}

The point here is that 
the right-hand side reserves $D(L)$ which measures the size of $L$,  
and that except $g_{sp}^{+}(K)^{-1}$ it depends only on $L$ and $D$ but not on $K$.

The estimate of ${\HM}(L, K)$ is thus shifted to that of $\prod_{p} \varepsilon_{p,\nu(p)}(L)$. 
Recall that what we finally need to estimate is not single ${\HM}(L, K)$ but rather their sum over all reflective vectors up to ${\Or}(L)$.  
Accordingly, we shall not estimate single $\prod_{p} \varepsilon_{p,\nu(p)}(L)$ but rather their following combination 
which will arise in the summation process (\S \ref{ssec:vol ratio II split}).  

\begin{definition}\label{def: e(L)}
Let $L$ be a lattice of signature $(2, n)$. 
For $p|2D(L)$ we put 
\begin{equation*}
{\epL} = \sum_{j, L_{p,j}\ne0} {\epjL}. 
\end{equation*}
Then we set 
\begin{equation*}
{\eL} 
= \prod_{p|D(L)}{\epL}
= \sum_{J}\left( \prod_{p|D(L)}\varepsilon_{p,j(p)}(L) \right),  
\end{equation*}
where 
$J = (j(p))_{p|D(L)}$ runs through multi-indices such that $L_{p,j(p)}\ne0$ for every $p$. 
Note that when $2\nmid D(L)$, we have $\varepsilon_2(L)=1$. 
\end{definition}

\textit{From now on we assume that $L$ is primitive.} 
The main step in the proof of Theorem \ref{branch obstruction} is the following.

\begin{proposition}\label{e(L) bdd split}
For primitive lattices $L$ the numbers ${\eL}$ are bounded in $n\geq4$: 
there exists a constant $\varepsilon < \infty$ independent of $L$ and $n$ such that ${\eL}\leq\varepsilon$ 
for every primitive lattice $L$ of signature $(2, n)$ with $n\geq4$.  
\end{proposition}

This proposition will not be used until Proposition \ref{prop: vol split final}, 
but we want to give the proof here because it would not be easy to remember $\varepsilon(L)$. 
In the proof the following easy estimate of $m_{p,j}(L)$ will be used several times. 
   
\begin{lemma}\label{lem:estimate m_pj(L)}
If $L$ is primitive, we have 
\begin{equation*}\label{eqn:estimate m_pj(L) 1}
m_{p,j}(L) \geq {\max} (0, \; n-n_{p,j}(L)). 
\end{equation*}
\end{lemma}

\begin{proof}
(See also Figure \ref{figure:m_{p,j}}.) 
Note that $L_{p,0}\ne0$ by the primitivity of $L$, and $L_{p,\mu(p)}\ne0$ by the definition of $\mu(p)$. 
We have 
\begin{eqnarray*}
& & m_{p,j}(L) \\  
& = & 
j \, (n_{p,0}(L)-1) + (\mu(p)-j) (n_{p,\mu(p)}(L)-1) + \sum_{k\ne0, j, \mu(p)} | k-j | \, n_{p,k}(L) \\ 
& \geq & 
\sum_{k\ne j} n_{p,k}(L) - 2 \\
& = & 
n -n_{p,j}(L). 
\end{eqnarray*}
The inequality $m_{p,j}(L)\geq 0$ is clear from the second line. 
\end{proof}
  
\begin{proof}[(Proof of Proposition \ref{e(L) bdd split})]
Since we will not change the lattice $L$ through the argument, let us abbreviate 
$n_{p,j}(L)=n_{p,j}$, $m_{p,j}(L)=m_{p,j}$ and $\varepsilon_{p,j}(L)=\varepsilon_{p,j}$. 
We divide the set of prime divisors of $D(L)$ into the following six sets, some of which could be empty: 
\begin{eqnarray*}
P_1  & = &  \{ 2 \}, \\ 
P_2  & = &  P, \\ 
P_3  & = &  \{ \: p>2 \; | \; \exists j \: \: n_{p,j}=n+1 \:\}, \\  
P_4  & = &  \{ \: p>2 \; | \; \exists j \: \: n_{p,j}=n \: \}, \\  
P_5  & = &  \{ \: p>2 \; | \; \forall   j \: \: n_{p,j}<n \: \: \textrm{and} \: \: \exists j \: \: n_{p,j} >  n/2+1\: \}, \\  
P_6  & = &  \{ \: p\not\in P\cup \{ 2\} \; | \; \forall j \: \: n_{p,j} \leq n/2+1 \: \}.   
\end{eqnarray*}
We will show that for each $P_i$, 
there exists a constant $\varepsilon(i)<\infty$ independent of $L$ and $n$ such that $\prod_{p\in P_i}{\epL} \leq \varepsilon(i)$. 
Then our assertion follows by putting $\varepsilon = \prod_{i=1}^{6}\varepsilon(i)$. 

($P_1$) There exists at most one index $j$ such that $n_{2,j}>n/2+1$. 
We have $\varepsilon_{2,j} \leq 1$ for this index. 
For the remaining indices $j$ we have $n_{2,j} \leq n/2+1$, so $m_{2,j} \geq n/2-1$ by Lemma \ref{lem:estimate m_pj(L)}, 
hence $\varepsilon_{2,j} \leq 2^{(2-n)/4}$. 
Since there are at most $n+2$ indices $j$ with $L_{2,j}\ne0$, we obtain 
\begin{equation*}
\varepsilon_2(L) < 1 + (n+2)2^{(2-n)/4}. 
\end{equation*}
Since $(n+2)2^{(2-n)/4}$ converges to $0$ as $n\to\infty$, the number 
\begin{equation*}
\varepsilon(1) = \max_{n\geq3} ( 1 + (n+2)2^{(2-n)/4} )
\end{equation*}
is finite, and we have $\varepsilon_2(L) < \varepsilon(1)$. 

($P_2$) 
If $p\in P$, we have $m_{p,j} \geq (n^2-1)/4$ by calculating the definition of $m_{p,j}$, 
and thus ${\epL}\leq 4(n+2)p^{(1-n^2)/8}$. 
It follows that 
\begin{equation*}
\prod_{p\in P} {\epL} \leq \prod_{p>2} {\max} (4(n+2)p^{(1-n^2)/8}, 1). 
\end{equation*}
For fixed $n$ there are only finitely many $p$ such that $4(n+2)p^{(1-n^2)/8}>1$, so  the right-hand side is actually a finite product. 
When $n\geq6$ we have $4(n+2)p^{(1-n^2)/8}<1$ for any $p>2$, so this product gets equal to $1$. 
Therefore 
\begin{equation*}
\varepsilon(2) = \max_{n\geq3} \left( \prod_{p>2} {\max}(4(n+2)p^{(1-n^2)/8}, 1)  \right)
\end{equation*}
is finite, and we have $\prod_{p\in P}{\epL} \leq \varepsilon(2)$. 

($P_3$) 
For primes $p$ in $P_3$, we have $(n_{p,0}, n_{p, \mu(p)})= (1, n+1)$ or $(n+1, 1)$, and $n_{p,j}=0$ for other indices $j$. 
We have $(m_{p,0}, m_{p, \mu(p)})=(n\mu(p), 0)$ and $(0, n\mu(p))$ in the respective cases, so 
\begin{equation*}
{\epL} = (1+p^{-[(n+1)/2]}) + 2p^{-n\mu(p)/2} \leq 1+3p^{-2}. 
\end{equation*}
If we put 
\begin{equation*}
\varepsilon(3) = \prod_{p>2} (1+3p^{-2}), 
\end{equation*}
we have $\prod_{P_3}{\epL} < \varepsilon(3)$ because every factor of $\varepsilon(3)$ is larger than $1$. 
When $p\geq11$, we have $1+3p^{-2}<1+p^{-3/2}$, 
so $\varepsilon(3)$ is dominated by some multiple of $\zeta(3/2)$, hence finite. 

($P_4$) 
There are three possibilities: 
\begin{enumerate}
\item $(n_{p,0}, n_{p,\mu(p)})= (2, n)$ or $(n, 2)$, and $n_{p,j}=0$ for all other $j$; 
\item $(n_{p,0}, n_{p,\mu(p)})= (1, n)$ or $(n, 1)$, and $n_{p,j}=1$ for some $0<j<\mu(p)$. 
\item $(n_{p,0}, n_{p,\mu(p)})= (1, 1)$, and $n_{p,j}=n$ for some $0<j<\mu(p)$; 
\end{enumerate}
In case (1), we have 
\begin{eqnarray}\label{eqn:bound P4 (1)}
{\epL} 
& = & p^{-\mu(p)/2}(1+p^{-[n/2]}) + p^{(1-n)\mu(p)/2}(1+p^{-1}) \nonumber \\ 
& \leq & p^{-1/2}(1+p^{-2}) + p^{-3/2}(1+p^{-1}). 
\end{eqnarray}
In case (2), we have $m_{p,k}\geq1$ for $k$ with $n_{p,k}=n$, and $m_{p,k}\geq n-1$ for $k$ with $n_{p,k}=1$. 
Hence 
\begin{equation}\label{eqn:bound P4 (2)}
{\epL}  \leq  p^{-1/2}(1+p^{-[n/2]}) + 4p^{(1-n)/2} 
 \leq  p^{-1/2}(1+p^{-2}) + 4p^{-3/2}. 
\end{equation}
In case (3), we have $m_{p,j}=0$ for $j$ with $n_{p,j}=n$, and $m_{p,0}, m_{p,\mu(p)} \geq n$. 
Therefore 
\begin{equation}\label{eqn:bound P4 (3)}
{\epL}  \leq  (1+p^{-[n/2]}) + 4p^{-n/2} \leq  1+5p^{-2}. 
\end{equation}
We have the bounds \eqref{eqn:bound P4 (1)}, \eqref{eqn:bound P4 (2)}, \eqref{eqn:bound P4 (3)} in the respective cases, 
but actually $1+5p^{-2}$ is greater than other two bounds. 
Therefore 
\begin{equation*}
{\epL}\leq 1+5p^{-2} 
\end{equation*} 
in any case. 
If we put 
\begin{equation*}
\varepsilon(4) = \prod_{p>2}(1+5p^{-2}), 
\end{equation*}
we have $\prod_{P_4}{\epL}<\varepsilon(4)$. 
Since $1+5p^{-2}<1+p^{-3/2}$ in $p\geq29$, $\varepsilon(4)$ is dominated by a multiple of $\zeta(3/2)$ and hence finite. 

($P_5$) 
We must have $n\geq5$ in this case. 
There exists only one index $j$ with $n_{p,j} > n/2+1$, for which we have $m_{p,j}\geq1$ by Lemma \ref{lem:estimate m_pj(L)} 
and hence $\varepsilon_{p,j}\leq p^{-1/2}(1+p^{-2})$. 
There remain at most $(n+1)/2$ indices $j$ with $L_{p,j}\ne0$. 
For them we have $m_{p,j} > n/2$, so $\varepsilon_{p,j} < 2p^{-n/4}$. 
It follows that 
\begin{equation*}
{\epL} < p^{-1/2}(1+p^{-2}) + (n+1)p^{-n/4}. 
\end{equation*}
As in the ($P_2$) case, there are only finitely many pairs $(n, p)$ such that the right-hand side is greater than $1$. 
Therefore 
\begin{equation*}
\varepsilon(5) = \max_{n\geq5} \left( \prod_{p>2} {\max}(p^{-1/2}(1+p^{-2}) + (n+1)p^{-n/4}, \: 1) \right)
\end{equation*}
is finite, and we have $\prod_{P_5}{\epL} < \varepsilon(5)$. 

($P_6$) 
By Lemma \ref{lem:estimate m_pj(L)} we have $m_{p,j}\geq n/2-1$ 
and so $\varepsilon_{p,j}\leq 2p^{(2-n)/4}$ for every index $j$ with $L_{p,j}\ne0$. 
Thus ${\epL}\leq 2(n+1)p^{(2-n)/4}$. 
As before 
\begin{equation*}
\varepsilon(6) = \max_{n\geq4} \left( \prod_{p>2} {\max}(2(n+1)p^{(2-n)/4}, \: 1) \right)
\end{equation*}
is finite, and we have $\prod_{P_6}{\epL} \leq \varepsilon(6)$. 
The proof of Proposition \ref{e(L) bdd split} is now finished. 
\end{proof}

\begin{remark}\label{remark:epsilon(L) bound}
(1) We needed the condition $n\geq4$ only in the $(P_4)$-(3) case. 
In other cases the boundedness can be easily extended to $n=3$. 

(2) 
In the proof we actually gave a bound at each $n$, say $\varepsilon(i, n)$, 
and $\varepsilon(i)$ was defined as ${\max}_n(\varepsilon(i, n))$. 
It would be useful to record the explicit form of $\varepsilon(i, n)$. 
Avoiding small $n$ and sharpening the estimate for $p=2$, 
we may take the bound as follows.  
\begin{eqnarray*}
\varepsilon(1, n) & = & 1 + 2^{-n/2+1} \qquad (n\geq 14), \\  
\varepsilon(2, n) & = &  1\qquad (n\geq 6), \\  
\varepsilon(3, n) & = &  \prod_{p>2}(1+3p^{-n/2}) < \prod_{p>2}(1+p^{-n/2+1}), \\  
\varepsilon(4, n) & = &  \prod_{p>2}(1+5p^{-[n/2]}) < \zeta([n/2]-2), \\  
\varepsilon(5, n) & = &  1 \qquad (n\geq 14), \\  
\varepsilon(6, n) & = &  1 \qquad (n\geq 16). 
\end{eqnarray*}
In particular, the total bound satisfies 
\begin{equation*}
\prod_{i=1}^{6}\varepsilon(i, n) < \zeta([n/2]-2)^2 
\end{equation*}
in $n\geq16$, so $\varepsilon$ can be taken to be asymptotically $1$. 
There is still room of improvement 
(by refining the classification by ${\max}_{j}(n_{p,j})$ and the number of $j$ with $n_{p,j}\ne0$), 
but we stop here.  

(3) 
By a similar argument as in case $(P_1)$, we can see that 
${\epL} \leq 1 + 2(n+2)p^{(2-n)/4}$ for $p\not\in P\cup \{ 2\}$.  
The product $\prod_{p}(1 + 2(n+2)p^{(2-n)/4})$ converges at each $n\geq7$ and is bounded with respect to $n$. 
This gives a simpler proof in $n\geq7$. 
\end{remark}


\subsection{Non-split case}\label{ssec:vol ratio I nonsplit}

Next we consider the non-split case. 
Let $L$ be a lattice of signature $(2, n)$ with $n\geq2$. 
Let $l\in L$ be a reflective vector of non-split type. 
The sublattice 
\begin{equation*}\label{eqn:def L' in vol estimate}
L' = {\Z}l \oplus K \qquad \textrm{where} \; \; \; K=l^{\perp}\cap L, 
\end{equation*}
is of index $2$ in $L$. 
The vector $l$ is reflective of split type in $L'$. 
Hence the definitions and results in \S \ref{ssec:vol ratio I split} before Proposition \ref{e(L) bdd split} are valid for $(L', K)$. 
Our approach is to reduce the estimate of the sum of ${\HMp}(L, K)$ of \textit{non-split} type for $L$   
to that of ${\HMp}(L', K)$ of \textit{split} type for $L'$ over various $L'\subset L$. 
This reduction step will be done in \S \ref{ssec:vol ratio II nonsplit}. 
Here we prepare in advance the counterpart of Proposition \ref{e(L) bdd split}.

We assume that $L$ is primitive and estimate $\varepsilon(L')=\prod_{p}\varepsilon_{p}(L')$. 
(In many cases $L'$ remains primitive, but not always.) 
When $p>2$, we have $L\otimes{\Zp}=L'\otimes{\Zp}$ and hence $L'\otimes{\Zp}$ is primitive.

\begin{lemma}\label{lem:e(L') and e(L'')}
Assume that $L$ is primitive and write $L'=L''(2^{\rho})$ with $L''$ primitive. 
Then $\rho\leq2$ and $\varepsilon_2(L')=2^{\rho/2}\varepsilon_2(L'')$. 
\end{lemma}

\begin{proof}
We have $n_{2,k}(L'')=n_{2,k+\rho}(L')$ for every $k$. 
In particular, if we write $D(L')_2=2^{\mu(2)'}$ and $D(L'')_2=2^{\mu(2)''}$, then $\mu(2)'=\mu(2)''+\rho$. 
By the definition of $m_{2,j}$ we see that  
\begin{equation*}
m_{2,j}(L'')+\mu(2)''=m_{2,j+\rho}(L')+\mu(2)'. 
\end{equation*} 
Hence $m_{2,j}(L'')=m_{2,j+\rho}(L')+\rho$, and so 
$2^{\rho/2}\varepsilon_{2,j}(L'')=\varepsilon_{2,j+\rho}(L')$. 
This implies $\varepsilon_2(L')=2^{\rho/2}\varepsilon_2(L'')$. 

We next check $\rho\leq2$.  
By the primitivity of $L\otimes{\Z}_2$, there exist vectors $l, m\in L\otimes{\Z}_2$ such that $(l, m)\in{\Z}_2^{\times}$. 
Since $L'\otimes{\Z}_2\subset L\otimes{\Z}_2$ is of index $2$, 
$2l$ and $2m$ are contained in $L'\otimes{\Z}_2$, and satisfies $(2l, 2m)\in 4{\Z}_2^{\times}$. 
On the other hand, we must have $(l', m')\in 2^{\rho}{\Z}_2$ for all $l', m' \in L'\otimes{\Z}_2$. 
Therefore $\rho\leq2$. 
\end{proof}

\begin{proposition}\label{e(L') bdd nonsplit}
Let $L$ be a primitive lattice of signature $(2, n)$ with $n\geq4$, 
and let $L'={\Z}l\oplus K$ for a reflective vector $l\in L$ of non-split type. 
Then 
\begin{equation*}
\varepsilon(L') \leq 2\varepsilon
\end{equation*}
where 
$\varepsilon$ is the constant introduced in Proposition \ref{e(L) bdd split}. 
\end{proposition}

\begin{proof}
For $p>2$ we have $n_{p,j}(L')=n_{p,j}(L'')$ for every $j$,  so $\varepsilon_p(L')=\varepsilon_p(L'')$. 
By Lemma \ref{lem:e(L') and e(L'')} we have 
$\prod_{p|D(L')}\varepsilon_p(L') \leq 2\prod_{p|D(L'')}\varepsilon_p(L'')$. 
Then we can apply Proposition \ref{e(L) bdd split} to the primitive lattice $L''$. 
\end{proof}


\section{Volume sum}\label{sec:vol estimate II}

Single volume ratios have been estimated in \S \ref{sec:vol estimate I}. 
Next we take their sum over the sets ${\RI}$, ${\RII}$ of branch divisors of each type. 
The proof of Theorem \ref{branch obstruction} will be completed at the end of this section.   


\subsection{Split case}\label{ssec:vol ratio II split}

We first deal with reflective vectors of split type. 
Let $L$ be a lattice of signature $(2, n)$ with $n\geq3$. 
We will not assume primitivity of $L$ until Proposition \ref{prop: vol split final}. 
For each natural number $D$ dividing $D(L)$, we write $\mathcal{R}_{{\rm I}}^{+}(D)$ for 
the set of ${\Or}(L)$-equivalence classes of reflective vectors of split type of norm $-D$. 
Note that if we have a splitting $L\simeq \langle -D \rangle \oplus K$, then $D$ must divide $D(L)$. 
We thus have the division 
\begin{equation*}\label{eqn:division by norm}
{\RI} = \bigsqcup_{D|D(L)} \mathcal{R}_{{\rm I}}^{+}(D). 
\end{equation*}
We also denote by ${\RI}(D)$ the set of ${\rm O}(L)$-equivalence classes of reflective vectors of split type of norm $-D$. 
It is more convenient to work with ${\rm O}(L)$ than with ${\Or}(L)$. 

\begin{lemma}\label{lem:proper vs improper}
We have 
\begin{equation*}\label{eqn:+ relation split case}
\sum_{[l]\in \mathcal{R}_{{\rm I}}^{+}(D)} {\HMp}(L, K) = \sum_{[l]\in \mathcal{R}_{{\rm I}}(D)} {\HM}(L, K), 
\end{equation*}
where $K=l^{\perp}\cap L$ for $[l]\in\mathcal{R}_{{\rm I}}^{+}(D)$ or $\mathcal{R}_{{\rm I}}(D)$. 
\end{lemma}

\begin{proof}
We have a natural projection $\mathcal{R}_{{\rm I}}^{+}(D)\to{\RI}(D)$. 
The cardinality of the fiber over $[l]\in{\RI}(D)$ is at most $2$ and equal to 
\begin{equation*}\label{eqn:spinor difference}
[{\rm O}(L) : {\Or}(L)]/[{\rm O}(K) : {\Or}(K)]. 
\end{equation*}
Indeed, when ${\rm O}(L)={\Or}(L)$, we have $\mathcal{R}_{{\rm I}}^{+}(D)={\RI}(D)$ and also ${\rm O}(K)={\Or}(K)$; 
when ${\rm O}(L)\ne{\Or}(L)$, the fiber consists of one element if and only if ${\rm O}(L)\cdot l = {\Or}(L)\cdot l$, namely 
$\gamma(l)=l$ for some $\gamma \in {\rm O}(L)\backslash{\Or}(L)$. 
This is equivalent to ${\rm O}(K)\ne{\Or}(K)$. 
Now the claim follows by comparison with \eqref{eqn:vol ratio spinor diff}. 
\end{proof}

We first estimate $\sum_{{\RI}(D)} {\HM}(L, K)$ for each $D$, and next take their sum over all possible $D$. 
Two reflective vectors of split type are ${\rm O}(L)$-equivalent if and only if their orthogonal complements are isometric. 
Thus ${\RI}(D)$ is canonically identified with the set of isometry classes of lattices $K$ such that 
$K\oplus \langle -D\rangle \simeq L$. 
We consider division into genera: 
\begin{equation*}
{\RI}(D) = \bigsqcup_{\alpha=1}^{\kappa} {\RI}(D)_{\alpha}.  
\end{equation*} 
Each ${\RI}(D)_{\alpha}$ consists of isometry classes of lattices $K$ in the same genus. 

\begin{lemma}\label{lem:KK estimate}    
The number $\kappa$ of possible genera of $K$ is at most $9$.  
\end{lemma} 

\begin{proof}
Scaling $L$ if necessary, we may assume that $L$ (and hence $K$) is even. 
By Nikulin's theory \cite{Ni}, it suffices to show that, with the discriminant forms $A_L$ and $A_{\langle -D \rangle}$ fixed, 
the number of isometry classes of finite quadratic forms $A$ such that 
\begin{equation}\label{eqn:prf of KK estimate}
A_L \simeq A_{\langle -D \rangle} \oplus A 
\end{equation}
is at most $9$. 

For $p>2$, the $p$-component $A_p$ of $A$ is uniquely determined by this relation, 
as can be seen from Wall's canonical form for quadratic forms on $p$-groups (\cite{Wa}). 
Alternatively, one can also directly resort to the Witt cancelation for ${\Zp}$-lattices in $p>2$ 
(see \cite{Ki} Corollary 5.3.1). 

For $p=2$ we use Kawauchi-Kojima's invariants $\sigma_r$ (\cite{K-K}) of quadratic forms on $2$-groups. 
(Here we identify, as in \cite{Wa} Theorem 5, 
quadratic forms and symmetric bilinear forms with no direct summand of order $2$.)   
These invariants are defined for each positive integer $r\geq1$,  
and take values in the semigroup $({\Z}/8)\cup\{\infty\}$. 
They have the properties that for two such forms $B$, $B'$, 
(i) $\sigma_r(B\oplus B')=\sigma_r(B)+\sigma_r(B')$, and 
(ii) $B$ and $B'$ are isometric if and only if their underlying abelian groups are isomorphic and 
$\sigma_r(B)=\sigma_r(B')$ for every $r\geq1$. 
Furthermore, (iii) when the abelian group underlying $B$ is isomorphic to ${\Z}/2^k$, 
we have $\sigma_r(B)<\infty$ for $r\ne k+1$. 

Now, with $(A_L)_2$ and $(A_{\langle -D \rangle})_2$ fixed in \eqref{eqn:prf of KK estimate}, 
the abelian group underlying $A_2$ is uniquely determined. 
We have $\sigma_r((A_{\langle -D \rangle})_2)<\infty$ except for one value of $r$. 
At these $r$, $\sigma_r(A_2)$ is uniquely determined by 
$\sigma_r(A_2)=\sigma_r((A_L)_2)-\sigma_r((A_{\langle -D \rangle})_2)$. 
Hence the isometry class of $A_2$ is determined by the value of $\sigma_r(A_2)$ at the remaining one $r$. 
\end{proof}

Since ${\HM}({\rm O}(K))$ depends only on the genus of $K$, we see that   
\begin{equation*}
\sum_{{\RI}(D)} {\HM}(L, K) = \sum_{\alpha=1}^{\kappa} |{\RI}(D)_{\alpha}| \cdot {\HM}(L, K). 
\end{equation*}
If $K\in {\RI}(D)_{\alpha}$, we have 
\begin{equation*}
|{\RI}(D)_{\alpha}| \leq g_{sp}^{+}(K)
\end{equation*}
because proper spinor genus coincides with proper equivalence class, 
which is finer than isometry class. 
We now substitute Proposition \ref{prop:vol split single}. 
We set 
\begin{equation}\label{eqn:def f(n)}
f(n) = \frac{2^{7}\cdot 9 \cdot \pi^{n/2+1} \cdot \zeta([n/2]+1)}{\Gamma(n/2+1)}. 
\end{equation}
Then 
\begin{equation*}\label{eqn:sum over RI(D)}
\sum_{\mathcal{R}_{{\rm I}}(D)} {\HM}(L, K) <  
f(n) \cdot D(L)^{-1/2} \cdot \prod_{p|D(L)} \varepsilon_{p,\nu(p)}(L), 
\end{equation*}
where the indices $\nu(p)$ are defined by $D=\prod_{p}p^{\nu(p)}$.  

We finally take the sum over the set of possible norms $-D$. 
We can identify $D=\prod_{p}p^{\nu(p)}$ with the multi-index $( \nu(p) )_{p|D(L)}$. 
If ${\RI}(D)\ne \emptyset$, then $L_{p,\nu(p)}\ne 0$ at each $p$. 
Thus the set of possible norms $-D$ can be regarded as a subset of the set of multi-indices 
$J = ( j(p) )_{p|D(L)}$ 
such that $L_{p,j(p)}\ne0$ at each $p$. 
Since ${\epjL}>0$ for all $(p, j)$ with $p|D(L)$ and $L_{p,j}\ne0$, we obtain by adding redundant $J$ 
\begin{eqnarray*}
\sum_{D}\sum_{\mathcal{R}_{{\rm I}}(D)}{\HM}(L, K) 
& < & \sum_{D} f(n) \cdot D(L)^{-1/2} \cdot \prod_{p|D(L)} \varepsilon_{p,\nu(p)}(L) \\
& \leq & f(n) \cdot D(L)^{-1/2} \cdot \sum_{J} \prod_{p|D(L)} \varepsilon_{p,j(p)}(L) \\ 
& = & f(n) \cdot D(L)^{-1/2} \cdot {\eL}
\end{eqnarray*}
where ${\eL}$ is as defined in Definition \ref{def: e(L)}. 

Let us summarize the argument so far, which worked without assuming $L$ primitive. 
This will be used again in the next section. 

\begin{lemma}\label{lem:semifinal estimate split without primitivity}
Let $L$ be a lattice of signature $(2, n)$ with $n\geq3$. 
Then 
\begin{equation*}
\sum_{[l]\in \mathcal{R}_{{\rm I}}}{\HMp}(L, K) < f(n) \cdot {\eL} \cdot D(L)^{-1/2}. 
\end{equation*}
\end{lemma}

Now assuming primitivity of $L$ and that $n\geq4$, 
we obtain from Proposition \ref{e(L) bdd split} the final estimate in the split case.  

\begin{proposition}\label{prop: vol split final}
For a primitive lattice $L$ of signature $(2, n)$ with $n\geq4$ we have  
\begin{equation*}
\sum_{[l]\in \mathcal{R}_{{\rm I}}}{\HMp}(L, K) < f(n) \cdot \varepsilon \cdot D(L)^{-1/2}  
\end{equation*}
where $\varepsilon$ is the constant introduced in Proposition \ref{e(L) bdd split} and $f(n)$ is the function defined by \eqref{eqn:def f(n)}. 
\end{proposition}


\subsection{Non-split case}\label{ssec:vol ratio II nonsplit}

We next consider the non-split case. 
Let $L$ be a lattice of signature $(2, n)$ with $n\geq3$. 
Recall from \S \ref{ssec:vol ratio I nonsplit} that for a reflective vector $l\in L$ of non-split type, 
our approach is to reduce the calculation of ${\HMp}(L, K)$ to that of ${\HMp}(L', K)$ 
where $K=l^{\perp}\cap L$ and $L'={\Z}l \oplus K$. 
Let us denote 
\begin{equation*}
\Gamma_{L'} = {\Or}(L) \cap {\Or}(L'), 
\end{equation*}
the intersection considered inside ${\rm O}(L_{{\Q}}) = {\rm O}(L'_{{\Q}})$. 
If we abuse notation to write 
\begin{equation}\label{eqn:[O(L):O(L')]}
[{\Or}(L):{\Or}(L')] = [{\Or}(L):\Gamma_{L'}]/[{\Or}(L'):\Gamma_{L'}], 
\end{equation}
we have by the relation \eqref{eqn:HM cofinite subgrp} 
\begin{equation}\label{eqn:HMLK and HML'K}
{\HMp}(L, K) 
=  [{\Or}(L):{\Or}(L')]  \cdot {\HMp}(L', K). 
\end{equation}

Let $T$ be the set of index $2$ sublattices $L'$ of $L$ for which 
there exists a reflective vector $l$ of $L$ of non-split type such that $L'={\Z}l\oplus(l^{\perp}\cap L)$.  
We write $\mathcal{T}=T/{\Or}(L)$. 
For each $L'\in T$ let $R[L']$ be the set of vectors $l\in L'$ which is primitive in $L'$ and splits $L'$, namely $L'={\Z}l\oplus(l^{\perp}\cap L')$. 
We put $\mathcal{R}[L'] = R[L']/{\Or}(L')$. 
In other words, $\mathcal{R}[L']$ is $\mathcal{R}_{{\rm I}}$ for $L'$. 

\begin{lemma}\label{lem:sum vol(L,K) and vol(L',K)} 
We have 
\begin{equation}\label{eqn:sum vol(L,K) and vol(L',K)} 
\sum_{[l]\in\mathcal{R}_{{\rm II}}}{\HMp}(L, K) \leq 
\sum_{[L']\in\mathcal{T}} [{\Or}(L):\Gamma_{L'}] \left( \sum_{[l]\in \mathcal{R}[L']} {\HMp}(L', K) \right). 
\end{equation}
Here $K=l^{\perp}\cap L$ for $[l]\in\mathcal{R}_{{\rm II}}$ in the left-hand side, 
while $K=l^{\perp}\cap L'$ for $[l]\in\mathcal{R}[L']$ in the right-hand side. 
\end{lemma} 

\begin{proof}
For each $L'\in T$, let $R'[L']\subset R[L']$ be the subset consisting of splitting vectors $l$ of $L'$ such that 
$l$ is still primitive in $L$ and that $l^{\perp}\cap L = l^{\perp}\cap L'$. 
This is equal to the set of reflective vectors $l$ of $L$ of non-split type such that $L'={\Z}l\oplus(l^{\perp}\cap L)$. 
Thus the set of reflective vectors of $L$ of non-split type is divided as 
$\bigsqcup_{L'\in T}R'[L']$, 
according to which index $2$ sublattice is ${\Z}l\oplus (l^{\perp}\cap L)$.  
Taking quotient by ${\Or}(L)$, we obtain 
\begin{equation*}
\mathcal{R}_{{\rm II}} = \bigsqcup_{[L']\in\mathcal{T}} R'[L']/\Gamma_{L'}
\end{equation*}
because $\Gamma_{L'}<{\Or}(L)$ is the stabilizer of $L'$ in the ${\Or}(L)$-action on $T$. 
Hence ${\RII}$ can be embedded into the \textit{formal} disjoint union 
\begin{equation*}
\bigsqcup_{[L']\in\mathcal{T}} R[L']/\Gamma_{L'}. 
\end{equation*}
(Note that when considered as sets of vectors of $L$, the sets $R[L']$ may have overlap with each other.) 
By \eqref{eqn:HMLK and HML'K} we have 
\begin{eqnarray*}
\sum_{[l]\in {\RII}}{\HMp}(L, K) 
&  =  & 
\sum_{[l]\in {\RII}} [{\Or}(L):{\Or}(L')]\cdot {\HMp}(L', K) \\  
& \leq &   
\sum_{[L']\in \mathcal{T}} [{\Or}(L):{\Or}(L')] \left( \sum_{[l]\in R[L']/\Gamma_{L'}} {\HMp}(L', K) \right). 
\end{eqnarray*}
Here $K=l^{\perp}\cap L$ in the first line, while $K=l^{\perp}\cap L'$ in the second line. 
Consider the projection $R[L']/\Gamma_{L'}\to \mathcal{R}[L']$. 
Its fibers have at most $[{\Or}(L'):\Gamma_{L'}]$ elements, so we have 
\begin{equation*}
\sum_{[l]\in R[L']/\Gamma_{L'}} {\HMp}(L', K) \leq [{\Or}(L'):\Gamma_{L'}] \cdot \sum_{[l]\in \mathcal{R}[L']} {\HMp}(L', K). 
\end{equation*}
Then our assertion follows by recalling \eqref{eqn:[O(L):O(L')]}.   
\end{proof} 

We estimate the right-hand side of \eqref{eqn:sum vol(L,K) and vol(L',K)}. 
Recall that Lemma \ref{lem:semifinal estimate split without primitivity} is still valid for $L'$. 
This gives for each $[L']\in\mathcal{T}$
\begin{equation*}
\sum_{\mathcal{R}[L']} {\HMp}(L', K) 
<  f(n)\cdot \varepsilon(L') \cdot D(L')^{-1/2}  
\leq  f(n)\cdot \varepsilon(L') \cdot D(L)^{-1/2}.  
\end{equation*}
In the second inequality we have 
$D(L') \geq D(L)$ 
because $A_L$ is an index $2$ quotient of an index $2$ subgroup of $A_{L'}$. 

We now assume primitivity of $L$ and $n\geq4$. 
By Proposition \ref{e(L') bdd nonsplit} we have 
\begin{equation*}
\sum_{\mathcal{R}[L']} {\HMp}(L', K) < f(n)\cdot 2\varepsilon \cdot D(L)^{-1/2}. 
\end{equation*}
Since the right-hand side does not depend on $L'$, we obtain  
\begin{equation*}
\sum_{[l]\in{\RII}} {\HMp}(L, K) < 
 \left( \sum_{[L']\in \mathcal{T}} [{\Or}(L):\Gamma_{L'}] \right) \cdot f(n)\cdot 2\varepsilon \cdot D(L)^{-1/2}. 
\end{equation*}
Since $\Gamma_{L'}<{\Or}(L)$ is the stabilizer of $L'\in T$ in the ${\Or}(L)$-action on $T$, 
then $[{\Or}(L):\Gamma_{L'}]$ equals to the cardinality of the ${\Or}(L)$-orbit of $L'$ in $T$. 
Therefore  
\begin{equation*}
\sum_{[L']\in \mathcal{T}} [{\Or}(L):\Gamma_{L'}] = |T| < 2^{n+2}. 
\end{equation*}
We arrive at the final estimate in the non-split case. 

\begin{proposition}\label{prop: vol nonsplit final}
For a primitive lattice $L$ of signature $(2, n)$ with $n\geq4$ we have  
\begin{equation*}
\sum_{[l]\in \mathcal{R}_{{\rm II}}}{\HMp}(L, K) < 2^{n+3} \cdot f(n) \cdot \varepsilon \cdot D(L)^{-1/2} 
\end{equation*}
where $\varepsilon$ is the constant introduced in Proposition \ref{e(L) bdd split} 
and $f(n)$ is the function defined by \eqref{eqn:def f(n)}. 
\end{proposition}

The above method can be used to give estimate of more general sum $\sum_{l}{\HMp}(L, K)$ where $l$ runs over 
(up to ${\Or}(L)$) primitive vectors such that ${\Z}l\oplus (l^{\perp}\cap L)$ is of a fixed index in $L$. 


\subsection{Proof of Theorem \ref{branch obstruction}}\label{ssec:proof bigness}

We can now prove Theorem \ref{branch obstruction} by combining the estimates obtained so far. 
Let $L$ be a primitive lattice of signature $(2, n)$ with $n\geq4$. 
We put 
\begin{equation*}
g(n) =  f(n) \cdot (1+4^{n+2}) \cdot \varepsilon 
\end{equation*}
where $f(n)$ and $\varepsilon$ are as introduced in \eqref{eqn:def f(n)} and Proposition \ref{e(L) bdd split} respectively. 
By Propositions \ref{prop: vol split final} and \ref{prop: vol nonsplit final}, 
the left-hand side of \eqref{eqn:big via HM volume} is bounded as 
\begin{equation}\label{eqn:vol final}
\sum_{{\RI}}{\HMp}(L, K) + 2^{n+1}\cdot \sum_{{\RII}}{\HMp}(L, K)
\: < \: g(n) \cdot D(L)^{-1/2}.  
\end{equation}
By Proposition \ref{prop:big via HM volume}, the ${\Q}$-divisor $a\mathcal{L}-B/2$ is big if the inequality 
\begin{equation}\label{eqn:bigness in terms of D(L)}
g(n) \cdot (1+a^{-1})^{n-1} \cdot (n/2a) \leq \sqrt{D(L)} 
\end{equation}
holds. 

If we fix $n$, there are only finitely many primitive lattices $L$ whose $D(L)$ does not exceed this bound. 
Indeed, the discriminant is bounded by $|A_L|\leq D(L)^{n+1}$, 
and there are only finitely many lattices of fixed signature with bounded discriminant. 
Thus we obtain the finiteness at each fixed $n$. 
Next, when $n$ grows, the left-hand side of \eqref{eqn:bigness in terms of D(L)} converges to $0$ 
due to the rapid decay of the Gamma factor $\Gamma(n/2+1)^{-1}$ in $f(n)$. 
Therefore the inequality \eqref{eqn:bigness in terms of D(L)} holds for every primitive lattice $L$ when $n$ is sufficiently large. 
This completes the proof of Theorem \ref{branch obstruction}.

\section{Effective computation}\label{sec:explicit computation}

\subsection{Bound of $n$}\label{ssec:bound of n}

In this subsection we explicitly compute a bound of $n$ above which all ${\FL}$ is of general type. 
By \S \ref{sec:cusp obstruction}, we always have a nonzero ${\Or}(L)$-cusp form of weight $\leq n/2+11$. 
So we may take $a=n/2-11$ in \eqref{eqn:bigness in terms of D(L)}. 
Since $\varepsilon\to1$ (Remark \ref{remark:epsilon(L) bound} (2)) and 
$(1+a^{-1})^{n-1} \to e^2$ for this value of $a$, 
the resulting bound is asymptotically given by \eqref{eqn:estimate intro}. 
This is smaller than $1$ at least in $n\geq300$, which gives a first bound. 

We can improve this using Lemma \ref{Vinberg reduction}. 
In the following we assume that $L$ is a lattice of signature $(2, n)$ such that 
$(A_L)_p\simeq({\Z}/p)^{l_p}$ with $l_p \leq n/2+1$ for every $p$. 
It suffices to compute a bound of $n$ for such lattices. 
For them we can improve some part of \S \ref{sec:bigness setup} -- \S \ref{sec:vol estimate II} as follows. 

First, if $l\in L$ is reflective of non-split type, 
then ${\rm div}(l)=2^ab$ with $b$ odd and $a\leq1$. 
When $a=0$, we have 
$(A_K)_2\simeq{\Z}/2\oplus(A_L)_2$, 
$(A_{L'})_2\simeq{\Z}/2\oplus(A_K)_2$ and  
$[{\Or}(K):\Gamma_l] \leq 2^{l_2}$ 
by Lemma \ref{stabilizer split/nonsplit}. 
When $a=1$, we have $(A_K)_2\simeq{\Z}/4\oplus({\Z}/2)^{l_2-2}$ and $(A_{L'})_2\simeq({\Z}/4)^2\oplus({\Z}/2)^{l_2-2}$. 
The gluing element $x$ in $(A_K)_2$ satisfies $x=2y$ for every element $y$ of order $4$, so is ${\rm O}(A_K)$-invariant. 
Hence $\Gamma_l={\Or}(K)$. 
Thus the left-hand side of \eqref{eqn:big via HM volume} can be replaced by 
\begin{equation}\label{eqn:big via HM volume after Vinberg reduction}
\sum_{{\RI}}{\HMp}(L, K) + \sum_{{\RII}, a=1}{\HMp}(L, K) + 2^{l_2}\cdot \sum_{{\RII}, a=0}{\HMp}(L, K). 
\end{equation}

The spinor genera $g_{sp}^+(L)$, $g_{sp}^+(L')$, $g_{sp}^+(K)$ are always equal to $1$ by \cite{Ca} Theorem 11.1.5. 
Also the set $P$ is empty (for $L$ and also for $L'$). 
We will not touch on the estimates in Lemma \ref{lem:a_p and epj split} (1), (2). 
On the other hand, the bound \eqref{eqn:2^7bound} can be improved to $\leq4$ for $l$ of split type. 
For non-split type $l$, replacing $L$ by $L'$, the bound \eqref{eqn:2^7bound} can be sharpened to $\leq1$. 
Finally, we have 
\begin{equation*}
\varepsilon_2(L') =  2^{-(l_2+1)/2} + 2^{(l_2+1-n)/2} 
\end{equation*}
in the non-split case with $a=0$. 
In other cases we do not improve the estimate of $\varepsilon_p(L)$, $\varepsilon_p(L')$ in Remark \ref{remark:epsilon(L) bound} (2). 
(Note that $L'$ is primitive.) 
To sum up, writing  
\begin{equation*}
h(n) = 9\cdot \pi^{n/2+1}\cdot \zeta([n/2]-2)^3 / \Gamma(n/2+1), 
\end{equation*}
we have 
\begin{equation*}
\sum_{{\RI}}{\HMp}(L, K) \; < \; 4\cdot h(n)  \cdot D(L)^{-1/2}, 
\end{equation*}
\begin{equation*}
\sum_{\mathcal{R}[L']}{\HMp}(L', K)  \; < \;  h(n) \cdot D(L)^{-1/2} \qquad (a=1),  
\end{equation*}
and when $a=0$, 
\begin{eqnarray*}
2^{l_2} \cdot \sum_{\mathcal{R}[L']}{\HMp}(L', K) 
&<& (2^{(l_2-1)/2} + 2^{(3l_2+1-n)/2}) \cdot h(n)  \cdot  D(L)^{-1/2} \\   
& \leq & (2^{n/4} + 2^{n/4+2}) \cdot h(n)  \cdot  D(L)^{-1/2}.  
\end{eqnarray*}
Repeating the process in \S \ref{ssec:vol ratio II nonsplit}, we obtain 
\begin{equation*}
\eqref{eqn:big via HM volume after Vinberg reduction}  \; < \;  \tilde{h}(n) \cdot D(L)^{-1/2} 
\end{equation*}
where 
\begin{equation*}
\tilde{h}(n) = (4+ 2^{n+2} + 2^{5n/4+2} + 2^{5n/4+4}) \cdot h(n). 
\end{equation*}
Thus every ${\FL}$ is of general type when 
\begin{equation*}\label{eqn:vol final Vinberg reduction}
\tilde{h}(n) \cdot (1+a^{-1})^{n-1} \cdot (n/2a) \leq 1, \qquad a=n/2-11. 
\end{equation*}
This holds in $n\geq109$. 
When $n=108$, the left-hand side is still smaller $\sqrt{2}$, 
and the unimodular case is of general type by the next \S \ref{ssec:odd unimodular}. 
We thus obtain the bound stated in Theorem \ref{main}. 

It would be possible to improve the bound of $n$ by doing case-by-case refined estimate 
for lattices whose $D(L)$ is smaller than the uniform bound above.


\subsection{Example: odd unimodular lattice}\label{ssec:odd unimodular}

As an explicit example we work out the odd unimodular lattices $I_{2,n}= 2\langle 1 \rangle \oplus n\langle -1 \rangle$. 
The even unimodular case $II_{2,2+8m}$ is studied by Gritsenko-Hulek-Sankaran \cite{G-H-S3}, 
who proved that $\mathcal{F}_{II_{2,n}}$ is of general type in $n\geq42$. 

\begin{proposition}
The variety $\mathcal{F}_{I_{2,n}}$ is of general type when $n\geq39$. 
\end{proposition}

\begin{proof}
We work with the maximal even sublattice $L$ of $I_{2,n}$, which is isometric to 
\begin{equation*}
L \simeq 2U \oplus D_{n-2} \simeq 2U \oplus mE_8 \oplus D_N, \qquad 1\leq N \leq 8.  
\end{equation*}
By convention, $D_1=\langle -4\rangle$ and $D_2=2A_1$. 
The case $N=1$ is treated in \cite{G-H-S3}, where ${\FL}$ is shown to be of general type in $m\geq5$. 
We consider the remaining case $N\geq2$. 
The discriminant form $A=A_{D_{N}}$ is as follows. 
We write $\langle \varepsilon / 2^{\mu} \rangle$ for the quadratic form on ${\Z}/2^{\mu}$ 
for which the standard generator has norm $\varepsilon / 2^{\mu}$ modulo $2{\Z}$. 
\begin{itemize}
\item If $N$ is odd, $A\simeq \langle -N/4 \rangle$; 
\item if $N=\pm2 \; (8)$, $A\simeq \langle \mp1/2 \rangle \oplus \langle \mp1/2 \rangle$; 
\item if $N=4$, $A=({\Z}/2)^{\oplus2}=\langle x_1, x_2 \rangle$ with $(x_i, x_i)=1$ and $(x_1, x_2)=1/2$; 
\item if $N=8$, $A\simeq A_{U(2)}$. 
\end{itemize}
Hence ${\Or}(I_{2,n})={\Or}(L)$ when $N\ne4$ and $[{\Or}(L)\colon{\Or}(I_{2,n})]=3$ for $N=4$. 

One can work out the general dimension formula in \cite{Sk}, \cite{Bo2} for $\rho_{A}^{{\rm O}(A)}$-valued cusp forms.  
This gives for $l>2$ with $l+N/2\in 2{\Z}$ 
\begin{equation*}
{\dim} S_l(\rho_{A})^{{\rm O}(A)} = 
{\dim} S_l(\rho_{A}^{{\rm O}(A)}) = 
\begin{cases}
[(2l+N)/8]-1 & N: \textrm{odd}, \\ 
[(l-2)/4]        & N=2, \\ 
[(l-2)/6]        & N=4, \\ 
[l/4]            & N=6, \\ 
[l/4]-1         & N=8. \\ 
\end{cases}
\end{equation*}
The minimal weight $l$ of ${\rm O}(A)$-invariant cusp forms is as in Table \ref{table l_0 D_N}.  
\begin{table}[h]
\caption{}\label{table l_0 D_N}
\begin{center} 
\begin{tabular}{c|c|c|c|c|c|c|c}
$N$   & $2$ & $3$      & $4$   & $5$      & $6$ & $7$    & $8$    \\ \hline  
$l$    & $7$ & $13/2$ & $8$   & $11/2$ & $5$ & $9/2$ & $8$     \\ \hline  
\end{tabular}
\end{center}
\end{table}



Next we calculate the branch obstruction. 
Let $e, f$ be the hyperbolic basis of $U$ and 
$\delta_1, \cdots, \delta_N$ the root basis of $D_N$ with 
$(\delta_1, \delta_2)=0$, $(\delta_1, \delta_3)=1$ and $(\delta_i, \delta_{i+1})=1$ for $i\geq2$. 
Then 
$l_1 = e-f$ and $l_2= \delta_1 - \delta_2$ 
are reflective vectors of non-split type of norm $-2$, $-4$ respectively. 
When $N=2$, we also have the splitting $(-2)$-vector $l_3=\delta_1$. 
If we write $K_i=l_i^{\perp}\cap L$, then 
\begin{equation*}
K_1 \simeq \langle2\rangle \oplus U \oplus D_N \oplus mE_8, \quad 
\end{equation*}
\begin{equation*}
K_2 \simeq 2U \oplus D_{N-1}\oplus mE_8, \quad 
\end{equation*}
\begin{equation*}
K_3 \simeq 2U \oplus A_1\oplus mE_8. 
\end{equation*}
By the Eichler criterion (\cite{Sc}), 
every reflective vector of $L$ is ${\Or}(L)$-equivalent to one of $l_1, l_2, l_3$. 
The stabilizer $\Gamma_i$ of $l_i$ coincides to ${\Or}(K_i)$ when $(i, N)\ne(1, 6), (2, 5)$. 
In those exceptional cases, $[{\Or}(K_i):\Gamma_i]=3$. 
The volume ratio ${\HMp}(L, K_i)$ is calculated as follows: 
\begin{equation*}
\begin{array}{ccc}
    &       i=1        & i=2               \\  
 &  & \\  
N=2  & \displaystyle \frac{\pi \cdot (2\pi)^{4m+2} \cdot (1-2^{-8m-4})}{(4m+2)! \cdot L(4m+3, \chi_{-4})}      
         & \displaystyle \frac{\pi^{4m+3} \cdot (1+2^{-4m-2})}{(4m+2)! \cdot L(4m+3, \chi_{-4})}               \\ 
 &  & \\  
N=3 & \displaystyle \frac{2^{4m+9/2} \cdot (4m+3)!\cdot L(4m+3, \chi_{-8})}{\pi^{4m+3} \cdot (1-2^{-4m-3}) \cdot B_{8m+6}}      
         & \displaystyle \frac{2 \cdot (4m+3)! \cdot L(4m+3, \chi_{-4})}{\pi^{4m+3} \cdot (1-2^{-4m-3}) \cdot B_{8m+6}}               \\     
 &  & \\  
N=4     & \displaystyle \frac{(1+2^{-4m-3}) \cdot (4m+4)}{(1-2^{-4m-4}) \cdot |B_{4m+4}|}      
           & \displaystyle \frac{3m+3}{2^{4m+1} \cdot (1-2^{-4m-4}) \cdot |B_{4m+4}|}               \\    
 &  & \\  
N=5     & \displaystyle \frac{2^{4m+11/2} \cdot (4m+4)! \cdot L(4m+4, \chi_{8})}{\pi^{4m+4} \cdot (1-2^{-4m-4}) \cdot |B_{8m+8}|}       
           & \displaystyle \frac{2^{4m+4} \cdot (1-2^{-4m-3}) \cdot B_{4m+4}}{3\cdot B_{8m+8}}               \\    
 &  & \\  
N=6     & \displaystyle \frac{\pi \cdot (2\pi)^{4m+4} \cdot (1-2^{-8m-8})}{3\cdot(4m+4)! \cdot L(4m+5, \chi_{-4})}      
           & \displaystyle \frac{\pi^{4m+5} \cdot (1-2^{-4m-4})}{(4m+4)! \cdot L(4m+5, \chi_{-4})}               \\  
 &  & \\  
N=7     & \displaystyle \frac{2^{4m+13/2} \cdot (4m+5)! \cdot L(4m+5, \chi_{-8})}{\pi^{4m+5} \cdot (1+2^{-4m-5}) \cdot B_{8m+10}}       
           & \displaystyle \frac{2\cdot (4m+5)! \cdot L(4m+5, \chi_{-4})}{\pi^{4m+5} \cdot (1+2^{-4m-5}) \cdot B_{8m+10}}               \\  
 &  & \\  
N=8     & \displaystyle \frac{(1-2^{-4m-5}) \cdot (4m+6)}{(1-2^{-4m-6}) \cdot B_{4m+6}}      
           & \displaystyle \frac{2m+3}{2^{4m+4} \cdot (1-2^{-4m-6}) \cdot B_{4m+6}}               
\end{array}
\end{equation*}       
and 
\begin{equation*}
{\HMp}(L, K_3) = \frac{\pi^{4m+3}}{2^{4m+1}\cdot (4m+2)! \cdot L(4m+3, \chi_{-4})}. 
\end{equation*}
Here $\chi_D(\cdot )=\left( \frac{D}{\cdot}\right)$ is the quadratic Kronecker symbol and $B_{2k}$ is the Bernoulli number. 
We insert these datum and $a=n/2+1-l$ into 
\begin{equation*}
\sum_{i} {\HMp}(\Gamma_i)/{\HMp}({\Or}(L)) \: < \: (1+a^{-1})^{1-n} (2a/n). 
\end{equation*}
The resulting inequality holds when $n\geq 39$. 
\end{proof}

Using quasi-pullback of Borcherds' $\Phi_{12}$ as in \cite{G-H-S1}, \cite{G-H-S4}, 
we can see that ${\FL}$ is of general type also in $n=23, 24$ 
(embed $D_N$ in $E_8$ with $D_N^{\perp}\simeq D_{8-N}$). 
On the other hand, ${\FL}$ is rational in $n\leq16$ and unirational in $n\leq20$. 
See \cite{Marational} for $n\leq18$;  
$L$ is the period lattice of quartic $K3$ surfaces in $n=19$, 
and of double EPW sextics in $n=20$ (\cite{OG}, \cite{G-H-S4}).


\appendix

\section{Singularity over 0-dimensional cusp}\label{sec:appendix}

Let $L$ be a lattice of signature $(2, n)$. 
Let $\Gamma$ be a finite-index subgroup of ${\Or}(L)$ and $\mathcal{F}(\Gamma)=\Gamma\backslash{\DL}$ the associated modular variety. 
For simplicity we assume $-1\in\Gamma$, which does not affect $\mathcal{F}(\Gamma)$. 

$0$-dimensional cusps of the Baily-Borel compactification of $\mathcal{F}(\Gamma)$ correspond to 
primitive isotropic vectors $l$ in $L$ up to the $\Gamma$-action. 
We write $M_l=l^{\perp}\cap L/{\Z}l$. 
Let $N(l)_{{\Q}}$ be the stabilizer of $l$ in ${\Or}(L_{{\Q}})$. 
The unipotent radical $U(l)_{{\Q}}$ of $N(l)_{{\Q}}$ consists of the Eichler transvections $E_{l,m}$, $m\in (M_l)_{{\Q}}$, 
which is defined by (cf.~\cite{Sc} \S 3.7)
\begin{equation*}
E_{l,m}(v) = v - (\tilde{m}, v)l + (l, v)\tilde{m} - \frac{1}{2}(m, m)(l, v)l, \qquad v\in L_{{\Q}},  
\end{equation*}
where $\tilde{m}\in l^{\perp}\cap L_{{\Q}}$ is a lift of $m$. 
Thus $U(l)_{{\Q}}$ is canonically identified with $(M_l)_{{\Q}}$. 
We have the fundamental exact sequence 
\begin{equation*}
0 \to U(l)_{{\Q}} \to N(l)_{{\Q}} \stackrel{\pi}{\to} {\Or}((M_l)_{{\Q}})\to 1. 
\end{equation*}
If we \textit{choose} a splitting $f\colon L_{{\Q}}\simeq U_{{\Q}} \oplus (M_l)_{{\Q}}$ with $f(l)\in U_{{\Q}}$, 
we obtain a section of $\pi$ and thus a non-canonical isomorphism  
\begin{equation}\label{eqn:non-canonical splitting of N(l)Q}
\varphi_{f} : N(l)_{{\Q}} \stackrel{\simeq}{\to} {\Or}((M_l)_{{\Q}})\ltimes U(l)_{{\Q}} = {\Or}((M_l)_{{\Q}})\ltimes (M_l)_{{\Q}}. 
\end{equation}
We write 
$N(l)_{{\Z}}=N(l)_{{\Q}}\cap \Gamma$, 
$U(l)_{{\Z}}=U(l)_{{\Q}}\cap \Gamma$ and 
$\overline{N(l)}_{{\Z}}= N(l)_{{\Z}}/U(l)_{{\Z}}$. 
For instance, when $\Gamma={\Ost}(L)$ with $L$ even, we have $U(l)_{{\Z}}=M_l$. 

Choose representatives $l_1, \cdots, l_N\in L$ of primitive isotropic vectors modulo $\Gamma$. 
We put a ${\Z}$-structure on $(M_i)_{{\R}}=(M_{l_i})_{{\R}}$ by $U(l_i)_{{\Z}}$. 
Let $\mathcal{C}_i$ be the union of the positive cone $(M_i)_{{\R}}^{+}$ of $(M_i)_{{\R}}$ and 
the rays ${\R}_{\geq0}m$ for $m\in (M_i)_{{\Q}}$ in the boundary of $(M_i)_{{\R}}^{+}$. 
According to \cite{AMRT}, toroidal compactification of $\mathcal{F}(\Gamma)$ can be constructed by choosing for each $i$ 
an $\overline{N(l)}_{{\Z}}$-admissible fan $\Sigma_i$ in $(M_i)_{{\R}}$ with $|\Sigma_i|=\mathcal{C}_i$. 
(There is no ambiguity of choice at the $1$-dimensional cusps, 
and the choices of fan at each $i$ are independent.) 
By \cite{AMRT}, we can choose $\Sigma_i$ to be regular with respect to $U(l_i)_{{\Z}}$. 

Our purpose in this appendix is to supplement a proof of the following 

\begin{theorem}[\cite{G-H-S1}]\label{thm:cano sing 0dim cusp}
When the fans $\Sigma_i$ are regular, 
the toroidal compactification $\mathcal{F}(\Gamma)^{\Sigma}$ associated to $\Sigma=(\Sigma_i)$ 
has canonical singularity at the points lying over the $0$-dimensional cusps. 
\end{theorem}

This theorem was first found by Gritsenko-Hulek-Sankaran (\cite{G-H-S1} \S 2.2), 
but as we explain later (Remark \ref{rmk:GHS0dimcusp}), their proof needs to be modified. 

Since Tai \cite{Ta}, proof of such a statement consists of the following steps: 
\begin{enumerate}
\item find a finite linear quotient model $V/G$ of the singularity; 
\item the Reid--Shepherd-Barron--Tai criterion \cite{Re}, \cite{Ta} tells whether $V/G$ has canonical singularity 
in terms of the eigenvalues of each element $g$ of $G$; 
\item so we are reduced to analyze $V$ as a representation of the cyclic group $\langle g\rangle$ for each $g\in G$. 
\end{enumerate}
In \S \ref{ssec:cyclic quotient in advance} we first present a certain class of representations $V$ of the cyclic groups ${\Z}/m$ 
and show that $V/({\Z}/m)$ has canonical singularity by the RST criterion. 
This part is elementary linear algebra and independent of modular varieties. 
We then study local model $V/G$ of the toroidal compactification and show (\S \ref{ssec:proof thm A1}) that for each $g\in G$, 
$V|_{\langle g\rangle}$ belongs to the class of representations we have studied in advance.

\subsection{Some cyclic quotients}\label{ssec:cyclic quotient in advance}

Let $G={\Z}/m$ be the standard cyclic group of order $m>1$. 
By a representation of $G$ we always mean a finite-dimensional complex representation. 
For $\mu\in \frac{1}{m}{\Z}/{\Z}$ we denote by $\chi_{\mu}$ the character $G\to{\C}^{\times}$ that sends $\bar{1}\in G$ to $e(\mu)$. 
For $d|m$ we write 
\begin{equation*}
V_d= \bigoplus_{k\in({\Z}/d)^{\times}}\chi_{k/d}. 
\end{equation*}
It is classical that a representation of $G$ defined over ${\Q}$ is isomorphic to $\oplus_{i}V_{d_i}$ for some $d_i|m$ (see \cite{Se} \S 13.1). 
When $m=m'm''$, we can view ${\Z}/m'$ as a subgroup of ${\Z}/m$ of index $m''$ by multiplication by $m''$: 
\begin{equation*}
{\Z}/m' \simeq m''{\Z}/m \subset {\Z}/m. 
\end{equation*}
If we put $d''=(d, m'')$ and $d'=d/d''$, the restriction of $V_d$ to ${\Z}/m'\subset {\Z}/m$ is isomorphic to 
a direct sum of copies of $V_{d'}$. 

If $d|m$ and $\mu\in \frac{1}{m}{\Z}/{\Z}$, we write $W_{d,\mu}$ for the $G$-representation 
\begin{equation*}
W_{d,\mu} = {\C}[{\Z}/d]\otimes \chi_{\mu} = \bigoplus_{k\in{\Z}/d} \chi_{k/d} \otimes \chi_{\mu}. 
\end{equation*}
Eigenvalues of $\bar{1}\in G$ on $W_{d,\mu}$ are the $e(\mu)$-shift of the $d$-th roots of $1$. 
Restriction rule is as follows. 

\begin{lemma}\label{lemma:restriction rule}
Let $m=m'm''$. 
We put $\mu'=m''\mu$, $d''=(d, m'')$ and $d'=d/d''$. 
The restriction of $W_{d,\mu}$ to ${\Z}/m' \subset {\Z}/m$ is isomorphic  to $(W_{d',\mu'})^{\oplus d''}$. 
\end{lemma}

\begin{proof}
We have $\chi_{\mu}|_{{\Z}/m'}=\chi_{\mu'}$. 
The image of ${\Z}/m'$ by the reduction map ${\Z}/m \to {\Z}/d$ is $d''{\Z}/d \simeq {\Z}/d'$,  
and ${\C}[{\Z}/d]|_{{\Z}/d'} \simeq {\C}[{\Z}/d']^{\oplus d''}$. 
\end{proof}

\begin{example}\label{example Wdmu}
Let $g\in {\rm GL}_{d}({\C})$ be the linear transformation 
\begin{equation*}
g = {\rm diag}(e(\alpha_1), \cdots, e(\alpha_d)) \circ (2, 3, \cdots, d, 1)
\end{equation*}
where $\alpha_i\in{\C}/{\Z}$. 
Let $m={\rm ord}(g)<\infty$. 
The eigenpolynomial of $g$ is $x^d-e(\sum_{i}\alpha_i)$. 
If $\mu\in{\Q}/{\Z}$ is an element with $d\mu=\sum_{i}\alpha_i$, 
it follows that ${\C}^d\simeq W_{d,\mu}$ as a representation of $\langle g \rangle \simeq {\Z}/m$. 
When $m=m'm''$, the restriction of the cyclic permutation $(2,\cdots, d, 1)$ to $\langle g^{m''} \rangle \simeq {\Z}/m'$ splits into 
$d''$ copies of cyclic permutation of length $d'$. 
In \S \ref{ssec:proof thm A1}, $W_{d,\mu}$ and Lemma \ref{lemma:restriction rule} will appear in this form.  
\end{example}

Based on Lemma \ref{lemma:restriction rule}, we make the following definition. 

\begin{definition}\label{def:admissible data}
Let $U$ be a representation of $G$ defined over ${\Q}$. 
Let $\{ (d_i, \mu_i) \}_i$ be a finite set of pairs $(d_i, \mu_i)$ with $d_i|m$ and $\mu_i\in \frac{1}{m}{\Z}/{\Z}$. 
We say that $\theta=(U, (d_i, \mu_i)_i)$ is an admissible data for $G$ if for every nontrivial subgroup $G'\simeq{\Z}/m'$ of $G$, either 
$U|_{G'}$ is nontrivial or 
$d_i':=d_i/(d_i, m'')>1$ for some $i$.  
\end{definition}

To such a data $\theta$ we associate the $G$-representation 
\begin{equation*}
V_{\theta} = U \oplus \bigoplus_{i} W_{d_i,\mu_i}. 
\end{equation*}
If we put 
\begin{equation}\label{eqn:restriction data}
\theta|_{G'} = ( U|_{G'}, ((d_i', \mu_i')^{\times d_i''})_i ) 
\end{equation}
for a subgroup $G'\simeq {\Z}/m'$ of $G$, Lemma \ref{lemma:restriction rule} shows that  
$V_{\theta}|_{G'}\simeq V_{\theta|_{G'}}$ as $G'$-representation. 
We have $(\theta|_{G'})|_{G''}=\theta|_{G''}$ for $G''\subset G' \subset G$. 
Hence admissibility of $\theta$ for $G$ implies that of $\theta|_{G'}$ for $G'$. 

Recall that a linear transformation of finite order is called \textit{quasi-reflection} (or \textit{pseudo-reflection}) 
if all but one of its eigenvalues are $1$. 

\begin{lemma}\label{lem:admi rep quasi-reflection}
Let $\theta=(U, (d_i, \mu_i)_i)$ be an admissible data for $G={\Z}/m$. 
Suppose that $G$ contains an element $g$ acting by quasi-reflection on $V_{\theta}$. 
Let $m'={\rm ord}(g)$ and $m''=m/m'$. 
Then $g$ acts on $V_{\theta}$ by reflection, so $m'=2$, and $m''$ is odd. 
The reflective vector $\delta\in V_{\theta}$ of $g$ is also an eigenvector of $G$, 
and contained in either $U$ or $W_{d_i,\mu_i}$ for some $i$. 
When $\delta\in U$, we have ${\C}\delta\simeq V_2$ as $G$-representation. 
When $\delta\in W_{d_i, \mu_i}$, we have $d_i=2$. 
\end{lemma} 

\begin{proof}
We can write $g=g_0^{m''}$ for a generator $g_0$ of $G$. 
There is only one eigenvalue $\lambda$ of $g_0$ such that $\lambda^{m''}\ne1$, 
and the remaining eigenvalues of $g_0$ are $m''$-th root of $1$. 
In particular, $\lambda$ has multiplicity $1$. 
Let $\delta$ be a generator of the $1$-dimensional $\lambda$-eigenspace of $g_0$. 
Since every eigenvalue of $g_0$ occurs in $U$ or one of $W_{d_i,\mu_i}$, 
the multiplicity one property implies that $\delta\in U$ or $\delta\in W_{d_i,\mu_i}$ for some $i$. 

First consider the case $\delta\in U$. 
Again by the multiplicity one, $\delta$ is contained in a sub $G$-representation isomorphic to $V_d$ for some $d|m$. 
Since $V_d|_{\langle g \rangle} \simeq (V_{d'})^{\oplus a}$ for $d'=d/(d, m'')$ while $g$ acts on this space by quasi-reflection, 
we must have $d'=2$ and $a=1$. 
Hence $d=2$, namely ${\C}\delta\simeq V_2$ as $G$-representation. 
Since $(-1)^{m''}=-1$, $m''$ is odd. 

Next consider the case $\delta\in W_{d_i,\mu_i}$. 
Since $g$ acts trivially on $U$ and $W_{d_j,\mu_j}$ for $j\ne i$, 
the admissibility condition says that we must have $d_i'>1$ in $W_{d_i,\mu_i}|_{\langle g \rangle}\simeq (W_{d_i',\mu_i'})^{\oplus d_i''}$. 
On the other hand, $g$ has only one $\ne1$ eigenvalue on $W_{d_i,\mu_i}$, 
so $d_i'=2$, $d_i''=1$ and $\mu_i'=0$ or $1/2$. 
Hence $d_i=2$ and $g$ acts by reflection. 
Since $W_{2,\mu_i}|_{\langle g\rangle} \simeq W_{2,\mu_i'}$, $m''$ is odd. 
\end{proof}

We can now present the main result of this subsection. 

\begin{proposition}\label{prop:admissible rep cano sing}
Let $\theta=(U, (d_i, \mu_i)_i)$ be an admissible data for $G={\Z}/m$. 
Then $V_{\theta}/G$ has canonical singularity. 
\end{proposition}

\begin{proof}
If $V$ is a representation of $G$ and $g\in G$ has eigenvalues $e(\alpha_1), \cdots, e(\alpha_n)$ with $0\leq \alpha_i <1$, 
the \textit{Reid-Tai sum} of $g$ is defined by 
\begin{equation*}
\Sigma_V(g) = \sum_{i=1}^{n} \alpha_i. 
\end{equation*}
(Similar invariant appears in the dimension formula for modular forms: see \cite{Sk}, \cite{Bo2}.) 
The Reid--Shepherd-Barron--Tai criterion \cite{Re}, \cite{Ta} says that 
when $G$ contains no quasi-reflection, 
$V/G$ has canonical singularity if and only if 
$\Sigma_V(g)\geq1$ for every $g\ne {\rm id} \in G$. 
We apply this to $V=V_{\theta}$ or its variation. 

We first consider the case $G$ contains no reflection on $V_{\theta}$. 

\begin{lemma}\label{lem:no reflection case agegeq1}
Let $\theta=(U, (d_i, \mu_i)_i)$ be an admissible data for $G=\langle g \rangle = {\Z}/m$. 
Assume that $g$ does not act as reflection on $V_{\theta}$. 
Then $\Sigma_{V_{\theta}}(g)\geq1$. 
\end{lemma}

\begin{proof}
Let $W=\bigoplus_i W_{d_i,\mu_i}$. 
It is clear that $\Sigma_{V_{\theta}}(g)\geq1$ in the following cases: 
\begin{itemize}
\item $U$ contains $V_d$ with $d\geq3$ or $(V_2)^{\oplus2}$; 
\item $W$ contains $W_{d,\mu}$ with $d\geq3$ or $W_{2,\mu}\oplus W_{2,\lambda}$; 
\item $U$ contains $V_2$ and $W$ contains $W_{2,\mu}$. 
\end{itemize}
The remaining cases are 
\begin{enumerate}
\item $U=V_2\oplus (V_1)^{\oplus a}$ and $W=\bigoplus_{i}W_{1,\mu_i}$; 
\item $U$ is trivial and $W=W_{2,\mu}\oplus \bigoplus_i W_{1,\mu_i}$. 
\end{enumerate}
In both cases $m$ must be even, say $m=2m'$. 
If $m'=1$, the eigenvalue $-1$ has multiplicity at least $2$ because $g$ is not reflection. 
Then $\Sigma_{V_{\theta}}(g)\geq1$. 
We show that the case $m'>1$ does not occur. 
Consider the restriction to the subgroup $G'=\langle g^2 \rangle\simeq {\Z}/m'$. 
Then $U|_{G'}$ is trivial. 
On the other hand, $W|_{G'}\simeq \bigoplus_i W_{1,2\mu_i}$ in case (1) and 
$W|_{G'}\simeq (W_{1,2\mu})^{\oplus2}\oplus \bigoplus_{i} W_{1,2\mu_i}$ in case (2) (in the sense of restriction in \eqref{eqn:restriction data}). 
By admissibility, we must have $m'=1$. 
\end{proof}

When $G$ contains no reflection, we can apply this lemma to all subgroups $G'$ of $G$ and their generators 
because $\theta|_{G'}$ is admissible for $G'$. 
By the RST criterion we obtain Proposition \ref{prop:admissible rep cano sing} in this case. 

We next consider the case $G$ contains an element $g$ acting as reflection on $V_{\theta}$. 
We may assume $G\ne\langle g\rangle$. 
Let $m''=m/2>1$ be the index of $\langle g \rangle$ in $G$, and $\delta$ a reflective vector of $g$. 
By Lemma \ref{lem:admi rep quasi-reflection}, $m''$ is odd, and $\delta$ is an eigenvector for $G$ contained in $U$ or some $W_{d_i,\mu_i}$. 
We write $\bar{G}<G$ for the subgroup of order $m''$. 
We have the decomposition $G=\bar{G}\oplus \langle g\rangle$ 
and $\bar{G}$ is canonically identified with $G/\langle g\rangle$. 
We set $\bar{V}=V_{\theta}/\langle g\rangle$, which is a $\bar{G}$-representation. 
We have $V_{\theta}/G\simeq \bar{V}/\bar{G}$, and we want to apply the previous step to $(\bar{V}, \bar{G})$. 
Note that $\bar{G}$ cannot contain reflection because its order $m''$ is odd. 

When $\delta\in U$, consider the $G$-decomposition $V_{\theta}=V'\oplus{\C}\delta$. 
By Lemma \ref{lem:admi rep quasi-reflection}, ${\C}\delta\simeq V_2$ as $G$-representation. 
Then as $\bar{G}$-representation 
\begin{equation*}
\bar{V} = V'\oplus ({\C}\delta)^{\otimes2} \simeq V'\oplus V_1 \simeq V_{\theta}. 
\end{equation*}
Since $\theta|_{\bar{G}}$ is admissible for $\bar{G}$, 
$\bar{V}/\bar{G}\simeq V_{\theta}/\bar{G}$ has canonical singularity by the previous step. 

When $\delta\in W_{d_i,\mu_i}$, we have $d_i=2$ by Lemma \ref{lem:admi rep quasi-reflection}. 
Since $W_{2,\mu_i}|_{\bar{G}}\simeq(W_{1,2\mu_i})^{\oplus2}$, then 
$\eta = (U, (d_j, \mu_j)_{j\ne i})|_{\bar{G}}$ must be admissible for $\bar{G}$. 
Hence $\Sigma_{V_{\eta}}(h)\geq1$ for every $h\ne {\rm id} \in \bar{G}$ by Lemma \ref{lem:no reflection case agegeq1}. 
Since $V_{\eta}$ is a direct summand of $\bar{V}$, we have $\Sigma_{\bar{V}}(h)\geq1$. 
Hence $\bar{V}/\bar{G}$ has canonical singularity. 
This finishes the proof of Proposition \ref{prop:admissible rep cano sing}. 
\end{proof}

\subsection{Toroidal compactification}\label{ssec:toroidal cpt 0dim cusp}

We go back to modular varieties and explain toroidal compactification over $0$-dimensional cusp. 
We keep the notation in the beginning of this appendix. 
Let $l\in L$ be a primitive isotropic vector and 
$\mathcal{D}_l = (M_l)_{{\R}}+i(M_l)_{{\R}}^{+}$ the tube domain associated to $l$. 
We \textit{choose} a vector $l'\in L_{{\Q}}$ with $(l, l')=1$ and identify $(M_l)_{{\Q}}$ with $\langle l, l' \rangle^{\perp}\cap L_{{\Q}}$. 
As explained in \S \ref{sec:convention}, this induces the tube domain realization 
\begin{equation*}\label{eqn:tube domain realization}
\iota_{l'} : \mathcal{D}_l \stackrel{\simeq}{\to} {\DL}, \qquad v\mapsto {\C}(l'+v-\frac{1}{2}((v, v)+(l', l'))l), 
\end{equation*}
which depends on $l'$. 
Via this, $U(l)_{\Q}\simeq (M_l)_{{\Q}}$ acts on $\mathcal{D}_l$ by parallel transformation. 
If we form the torus $T_l=(M_l)_{{\C}}/U(l)_{{\Z}}$, 
then $\iota_{l'}^{-1}$ maps $X_l = {\DL}/U(l)_{{\Z}}$ isomorphically to the open set 
$ \mathcal{D}_l/U(l)_{{\Z}} = {\rm ord}^{-1}( (M_l)_{{\R}}^{+})$ of $T_l$. 
The group $\overline{N(l)}_{{\Z}}$ acts on $X_l$ through the $N(l)_{{\Z}}$-action on ${\DL}$. 

The action of $N(l)_{{\Z}}$ on $U(l)_{{\Q}}\simeq (M_l)_{{\Q}}$ preserves the lattice $U(l)_{{\Z}}$. 
Hence if $\pi:N(l)_{{\Q}}\to{\Or}((M_l)_{{\Q}})$ is the natural map, 
$N(l)_{{\Z}}$ is contained in $\pi^{-1}({\Or}(U(l)_{{\Z}}))$, of which $U(l)_{{\Z}}$ is a normal subgroup. 
Thus $\overline{N(l)}_{{\Z}}$ is canonically a subgroup of $\pi^{-1}({\Or}(U(l)_{{\Z}}))/U(l)_{{\Z}}$. 
By \eqref{eqn:non-canonical splitting of N(l)Q}, the splitting 
$L_{{\Q}}=\langle l, l' \rangle_{{\Q}}\oplus (M_l)_{{\Q}}$ 
given by $l'$ induces an isomorphism 
\begin{equation*}
\varphi_{l'} : \pi^{-1}({\Or}(U(l)_{{\Z}}))/U(l)_{{\Z}} \to {\Or}(U(l)_{{\Z}})\ltimes (U(l)_{{\Q}}/U(l)_{{\Z}}). 
\end{equation*}
The right side group is canonically a subgroup of 
\begin{equation*}
{\rm GL}(U(l)_{\Z})\ltimes (U(l)_{{\Q}}/U(l)_{{\Z}}) = {\rm Aut}(T_l)\ltimes (T_l)_{tor} \subset {\rm Aut}(T_l)\ltimes T_l. 
\end{equation*}
We thus obtain an embedding depending on $l'$
\begin{equation*}
\varphi_{l'} : \overline{N(l)}_{{\Z}} \hookrightarrow {\rm Aut}(T_l)\ltimes T_l.  
\end{equation*}
By the definition of $\overline{N(l)}_{\Z}$, the projection $\varphi_{l'}(\overline{N(l)}_{{\Z}})\to{\rm Aut}(T_l)$ is injective. 
If we express $\varphi_{l'}(g)=(\gamma, a)\in{\rm Aut}(T_l)\ltimes T_l$ for $g\in \overline{N(l)}_{\Z}$, 
then $\gamma=\pi(\tilde{g})$ and $a=[\tilde{g}(l')-l']$ 
where $\tilde{g}\in N(l)_{{\Z}}$ is a lift of $g$.

The affine group ${\rm Aut}(T_l)\ltimes T_l$ acts on $T_l$ naturally: 
${\rm Aut}(T_l)$ by torus automorphisms (fixing the identity), and $T_l$ by translation. 
The $\overline{N(l)}_{{\Z}}$-action on $X_l$ 
is the restriction of the action of ${\rm Aut}(T_l)\ltimes T_l$ on $T_l$ through $\varphi_{l'}$ and $\iota_{l'}$.

\begin{remark}\label{rmk:GHS0dimcusp}
In \cite{G-H-S1} p.~534, Gritsenko-Hulek-Sankaran implicitly assume that 
$\varphi_{l'}(\overline{N(l)}_{{\Z}})$ is contained in ${\rm Aut}(T_l)$ for some $l'\in L_{{\C}}$ 
so that the translation component $a=a_g$ is trivial for every $g$. 
If this holds, $N(l)_{{\Z}}$ will decompose into $\overline{N(l)}_{{\Z}}\ltimes U(l)_{{\Z}}$. 
However, this assumption seems to be too strong in general. 
For each $g$, $a_g$ varies holomorphically with $l'$ so that it is not $1$ for generic $l'$, 
and it seems highly nontrivial or even impossible for general $\Gamma$ that one can find a specific $l'$ 
such that $a_g=1$ for all $g$. 
(Note that the isomorphism $\mathcal{D}_{L}(F)\simeq U(F)_{{\C}}$ in loc.~cit.~depends on 
the choice of a base point ${\C}\omega$ of $\mathcal{D}_{L}(F)$. 
This isomorphism is the extension of $\iota_{l'}$, 
and ${\C}\omega$ is another intersection point of ${\proj}\langle l, l'\rangle_{{\C}}$ with the isotropic quadric.) 

On the other hand, in the important example $\Gamma={\Ost}(L)$ with $L$ even, 
$\varphi_{l'}(\overline{N(l)}_{{\Z}})$ is indeed contained in ${\rm Aut}(T_l)$ if $l'$ is taken from $L^{\vee}$. 
Hence in this case the proof of \cite{G-H-S1} works. 
\end{remark}

Now let $\Sigma_l$ be the $\overline{N(l)}_{{\Z}}$-admissible regular fan in $(M_l)_{{\R}}$ we have chosen for $l$. 
This defines a torus embedding $T_l\hookrightarrow T_{\Sigma_l}$. 
The partial compactification $X_{\Sigma_l}$ of $X_l$ in the direction of $l$ is by definition the interior of the closure of $X_l$ in $T_{\Sigma_l}$. 
The group $\overline{N(l)}_{{\Z}}$ acts on $X_{\Sigma_l}$ properly discontinuously. 
We have a natural map 
\begin{equation*}
X_{\Sigma_l}/\overline{N(l)}_{{\Z}} \to \mathcal{F}(\Gamma)^{\Sigma}, 
\end{equation*}
which is locally isomorphic at the points lying over the $0$-dimensional cusp ${\C}l$ (\cite{AMRT} p.~175). 
Hence Theorem \ref{thm:cano sing 0dim cusp} reduces to the following assertion (cp.~\cite{G-H-S1} Theorem 2.17).  

\begin{theorem}\label{thm:cano sing toric}
Let $N$ be a free abelian group of finite rank and $T=T_N$ be the associated torus. 
Let $G$ be a finite subgroup of ${\rm Aut}(T)\ltimes T$ such that $G\to{\rm Aut}(T)$ is injective. 
Let $\Sigma$ be a regular fan in $N_{{\R}}$ preserved by $G$, and $T_{\Sigma}=T_{N,\Sigma}$ the torus embedding defined by $\Sigma$. 
Then $T_{\Sigma}/G$ has canonical singularity. 
\end{theorem}

In the next subsection we prove this by reducing it to Proposition \ref{prop:admissible rep cano sing}. 
Note that the injectivity condition on $G\to {\rm Aut}(T)$ is essential:  
consider the extreme situation $G\subset T$, where one loses control of the Reid-Tai sum.

\subsection{Proof of Theorem \ref{thm:cano sing toric}}\label{ssec:proof thm A1}

Let $x$ be a point of $T_{\Sigma}$ and $G_x\subset G$ be the stabilizer of $x$. 
It suffices to prove that $T_xT_{\Sigma}/G_x$ has canonical singularity. 
By the well-known cyclic reduction (\cite{Re}, \cite{Ta}), 
this reduces to showing that $T_xT_{\Sigma}/\langle g\rangle$ has canonical singularity for every $g\in G_x$. 
We write $m$ for the order of $g$. 
Let ${\rm orb}(\sigma)$ be the $T$-orbit $x$ belongs to, where $\sigma$ is a regular cone in $\Sigma$. 
Write $g=(\gamma, a)\in{\rm Aut}(T)\ltimes T$. 
Since $g$ preserves ${\rm orb}(\sigma)$, $\gamma$ preserves the cone $\sigma$, permuting its rays. 
The open embedding $T_{\sigma}\hookrightarrow T_{\Sigma}$ is $g$-equivariant, hence 
$T_xT_{\Sigma}=T_xT_{\sigma}$ as $\langle g\rangle$-representation. 
We are thus reduced to showing that $T_xT_{\sigma}/\langle g\rangle$ has canonical singularity. 

Since $g$ has finite order, we have the $g$-decomposition 
\begin{equation*}
T_xT_{\sigma} = T_x({\rm orb}(\sigma)) \oplus N_x({\rm orb}(\sigma)). 
\end{equation*}
Let $N_0={\Z}(\sigma\cap N)$ and $N_1=N/N_0$, which are free $\gamma$-modules. 
We have a natural isomorphism ${\rm orb}(\sigma)\simeq T_{N_1}$ so that $T_x({\rm orb}(\sigma))\simeq (N_1)_{{\C}}$. 
The rays of $\sigma$ define a basis of $N_0$, and $\gamma$ acts on $N_0$ by permuting these basis vectors. 
Let $(d_1, \cdots, d_N)$ be the cyclic type of this permutation ($\sum_i d_i={\rm rk}(N_0)$). 

\begin{proposition}
(1) Via the isomorphism $T_x({\rm orb}(\sigma))\simeq (N_1)_{{\C}}$, 
the $g$-action on $T_x({\rm orb}(\sigma))$ is identified with the $\gamma$-action on $(N_1)_{{\C}}$. 
In particular, it is defined over ${\Q}$. 

(2) As a representation of $\langle g\rangle\simeq{\Z}/m$, 
the normal space $N_x({\rm orb}(\sigma))$ is isomorphic to $\bigoplus_{i=1}^{N}W_{d_i,\mu_i}$ for some $\mu_1, \cdots, \mu_N\in{\Q}/{\Z}$. 

(3) The data $(T_x({\rm orb}(\sigma)), (d_i, \mu_i)_i)$ for $\langle g\rangle\simeq{\Z}/m$ 
is admissible in the sense of Definition \ref{def:admissible data}. 
\end{proposition}

Theorem \ref{thm:cano sing toric} follows from the assertion (3) and Proposition \ref{prop:admissible rep cano sing}. 

\begin{proof}
We first show that (3) follows from (1) and (2). 
Suppose we have a factorization $m=m'm''$ with $m'\ne1$ and consider the restriction of 
$((N_1)_{{\C}}, (d_i, \mu_i)_i)$ to the subgroup $\langle g^{m''} \rangle\simeq{\Z}/m'$ of  $\langle g\rangle\simeq{\Z}/m$. 
As explained in Example \ref{example Wdmu}, the restriction of the cyclic permutation $(2, \cdots, d_i, 1)$ 
to ${\Z}/m'\subset {\Z}/m$ splits into 
copies of $(2, \cdots, d_i', 1)$ where $d_i'=d_i/(d_i, m'')$. 
Therefore, if $d_i'=1$ for all $1\leq i \leq N$, the $\gamma^{m''}$-action on $N_0$ must be trivial. 
If furthermore $\gamma^{m''}$ acts on $N_1$ trivially, then $\gamma^{m''}={\rm id}$. 
By the injectivity of $\langle g\rangle \to {\rm GL}(N)$, we have $g^{m''}={\rm id}$, so $m'=1$. 
This shows that $((N_1)_{{\C}}, (d_i, \mu_i)_i)$ is admissible. 

We check (1). We write $T_1=T_{N_1}$. 
We have a canonical isomorphism $T_yT_1\simeq (N_1)_{{\C}}$ for every $y\in T_{N_1}$. 
Via this $\gamma :T_xT_1\to T_{\gamma x}T_1$ is identified with $\gamma :(N_1)_{{\C}}\to(N_1)_{{\C}}$, 
and the translation $t_a:T_{\gamma x}T_1\to T_xT_1$ with the identity of $(N_1)_{{\C}}$. 

We verify (2). 
We write $T_0=T_{N_0}$. 
Via the generators of the rays of $\sigma$, $T_0\subset (T_0)_{\sigma}$ is isomorphic to $({\C}^{\times})^r \subset {\C}^r$, 
and $\gamma$ acts on $(T_0)_{\sigma}\simeq{\C}^r$ by permuting the basis vectors. 
We have a canonical isomorphism $T_{\sigma}\simeq T\times_{T_0}(T_0)_{\sigma}$ 
which makes $T_{\sigma}$ a vector bundle over $T_1$ with zero section ${\rm orb}(\sigma)$. 
Let $\pi\colon T_{\sigma} \to T_1\simeq {\rm orb}(\sigma)$ be the projection. 
If $y\in T$, the $\pi$-fiber through $y$ gets isomorphic to $(T_0)_{\sigma}$ by 
\begin{equation*}
\varphi_y : \pi^{-1}(\pi(y)) \to (T_0)_{\sigma}, \qquad [(y, z)]\mapsto z. 
\end{equation*}
This trivialization depends on $y$: 
if we replace $y$ by $y'=b^{-1} y$ where $b\in T_0$, then $\varphi_{y'}\circ \varphi_{y}^{-1}$ acts on $(T_0)_{\sigma}$ 
by the torus action by $b$. 

Now take a point $y\in T$ with $\pi(y)=x$, the fixed point of $g=t_a\circ \gamma$ in question. 
Via $\varphi_y$ and $\varphi_{\gamma y}$ 
the map $\gamma\colon \pi^{-1}(x) \to \pi^{-1}(\gamma x)$ is identified with the permuting action of $\gamma$ on $(T_0)_{\sigma}$, 
and via $\varphi_{\gamma y}$ and $\varphi_y$ 
the map $t_a\colon \pi^{-1}(\gamma x) \to \pi^{-1}(x)$ with the torus action of an element of $T_0$ on $(T_0)_{\sigma}$. 
Via the trivialization $(T_0)_{\sigma}\simeq{\C}^r$, the last action is expressed by a diagonal matrix. 
Hence via $\varphi_y$ and $(T_0)_{\sigma}\simeq{\C}^r$, 
the map $g\colon \pi^{-1}(x)\to \pi^{-1}(x)$ is expressed by a direct sum of linear transformations of the form 
\begin{equation*}
{\rm diag}(e(\alpha_{?}), \cdots, e(\alpha_{?})) \circ (2, 3, \cdots, d_i, 1)
\end{equation*}
over $i=1, \cdots, N$. 
In view of Example \ref{example Wdmu}, this proves our assertion. 
\end{proof}

\subsection{No ramifying boundary divisor}

We keep the notation in \S \ref{ssec:toroidal cpt 0dim cusp}. 
In \cite{G-H-S1}, Gritsenko-Hulek-Sankaran also proved the following. 

\begin{proposition}
The natural projection $X_{\Sigma_l}\to \mathcal{F}(\Gamma)^{\Sigma}$ has no ramification divisor at the boundary. 
\end{proposition}

This is equivalent to saying that no nontrivial element of $\overline{N(l)}_{{\Z}}$ fixes a boundary divisor of $X_{\Sigma_l}$. 
By the same reason the proof of this assertion also needs to be modified, 
but this is easier than Theorem \ref{thm:cano sing 0dim cusp}. 
It suffices to check the following. 

\begin{lemma}
Let $N$ and $T$ be as in Theorem \ref{thm:cano sing toric}. 
Let $g=(\gamma, a)$ be a finite order element of ${\rm Aut}(T)\ltimes T$ such that $\gamma\ne {\rm id}$. 
Let $\sigma\subset N_{{\R}}$ be a ray fixed by $\gamma$. 
Then the $g$-action on $T_{\sigma}$ does not fix the boundary divisor ${\rm orb}(\sigma)$. 
\end{lemma}

\begin{proof}
Let $N_0={\Z}(\sigma\cap N)$ and $N_1=N/N_0$. 
Via the natural isomorphism ${\rm orb}(\sigma)\simeq T_{N_1}$, 
$g$ acts on ${\rm orb}(\sigma)$ by $t_{\bar{a}}\circ \bar{\gamma}$ 
where $\bar{a}\in T_{N_1}$ is the image of $a$ and $\bar{\gamma}$ is the $\gamma$-action on $N_1$. 
If this was identity, then $\bar{a}=1$ and $\bar{\gamma}={\rm id}$. 
Hence $\gamma$ acts on both $N_0$ and $N_1$ trivially, so $\gamma={\rm id}$. 
\end{proof}


\end{document}